\documentclass[11pt,letter]{article}
\usepackage[latin1]{inputenc}
\usepackage{fancyhdr}
\usepackage{indentfirst}
\usepackage{graphicx}
\usepackage{epstopdf}
\usepackage{color}
\usepackage{newlfont}
\usepackage{amssymb}
\usepackage{amsmath}
\usepackage{latexsym}
\usepackage{amsthm}
\usepackage{amscd}
\usepackage{fullpage}
\usepackage{multirow}
\usepackage{hyperref}
\hypersetup{
  linkcolor=red,          % color of internal links
  citecolor=cyan,        % color of links to bibliography
}

\theoremstyle{plain}                    
\newtheorem{theorem}{Theorem}[section]     
\newtheorem{lem}[theorem]{Lemma}            
\theoremstyle{definition}

\newtheorem{ese}[theorem]{Example}           
\newtheorem{remark}[theorem]{Remark}
\newtheorem{prop}[theorem]{Proposition}
\newtheorem{cor}[theorem]{Corollary}
\newtheorem*{maintheorem}{Main Theorem}

\newcommand{\de}{\mathrm{d}}
\newcommand{\virg}[1]{``#1''}
\newcommand{\N}{\mathbb{N}}
\newcommand{\Z}{\mathbb{Z}}
\newcommand{\Q}{\mathbb{Q}}

\newcommand{\bianco}{\textcolor{white}{.}}

\newcommand{\les}{\left[\left[}
\newcommand{\res}{\right]\right]}
\newcommand{\X}{(0,1]\smallsetminus\mathbb{Q}}
\input xy
\xyoption{all}
\begin{document}
\clearpage{\pagestyle{empty}\cleardoublepage}
\title{Renewal-type Limit Theorem for Continued Fractions\\ with Even Partial Quotients}
\author{Francesco Cellarosi\footnote{Mathematics Department, Princeton University. fcellaro@math.princeton.edu}}
%\date{February 28, 2008}
\date{August 5, 2008}
\maketitle
\begin{abstract}
We prove the existence of the limiting distribution for the sequence of denominators generated by continued fraction expansions with even partial quotients, which were introduced by F. Schweiger \cite{Schweiger82} \cite{Schweiger84} and studied also by C. Kraaikamp and A. Lopes \cite{Kraaikamp-Lopes96}. Our main result is proven following the strategy used by Ya. Sinai and C. Ulcigrai \cite{Sinai-Ulcigrai07} in their proof of a similar renewal-type theorem for Euclidean continued fraction expansions and the Gauss map. 
The main steps in our proof are the construction of a natural extension of a Gauss-like map and the proof of mixing of a related special flow.
\end{abstract}

\section*{Introduction}
Continued fractions with even partial quotients (ECF) were originally introduced by F. Schweiger \cite{Schweiger82} \cite{Schweiger84} and appeared to be qualitatively different from the usual
Euclidean continued fractions. In order to study ECF-expansions from a dynamical point of view, one should define a Gauss-like map $T$ and its associated jump transformation $R$, which is used to overcome the intermittent behavior of $T$. Let $\{q_n\}_{n\in\N}$ be the sequence of the denominators of the ECF-convergents. For $L>0$ define the renewal time $n_L=\min\{n\in\N:\: q_n>L\}$. Our main result is a renewal-type limit theorem which establishes the existence of a limiting distribution for $\frac{q_{n_L}}{L}$, jointly with any finite number of entries preceeding and following the renewal time $n_L$ when $L\rightarrow\infty$. 

With the help of the map $R$, we construct a subsequence $\{\hat q_n\}_{n\in\N}\subset\{q_n\}_{n\in\N}$ and define the renewal time $\hat n_L=\min\{n\in\N:\:\hat q_n>L\}$. The main result 
follows from another renewal-type theorem, which shows the existence of a limiting distribution for $\frac{\hat q_{\hat n_L}}{L}$. 
The proof of this theorem exploits the mixing properties of a suitably defined special flow over the natural extension of $R$, following
the strategy used by Ya. Sinai and C. Ulcigrai \cite{Sinai-Ulcigrai07}.

The paper is organized as follows: in Section \ref{section: preliminaries and main result} we introduce ECF-expansions comparing them with Euclidean continued fraction expansions. 
After introducing the maps $T$ and $R$ and the sequences of denominators $\{\hat q_n\}_{n\in\N}\subset\{q_n\}_{n\in\N}$, we 
formulate our two renewal-type theorems (concerning $\{q_n\}_{n\in\N}$ and $\{\hat q_n\}$ respectively) 
and present the main tools (natural extensions and special flows). 
Section \ref{section: recurrence relations and denominator estimates} is devoted to estimating the growth of the sequence $\{\hat q_n\}_{n\in\N}$.
In Section \ref{section: reduction to a special flow} we define a special flow over the natural extension of $R$ and 
show how it is 
used to approximate the logarithm of the $R$-denominators $\hat q_n$. 
The mixing property of the flow is shown in Appendix A. Further lemmata are proven in Section \ref{section: reduction to a special flow}, allowing us to ``localize'' the problem to sufficiently small cylinders. Section \ref{section: existence of limiting distribution} contains the proof of the renewal-type theorem for $\{\hat q_n\}_{n\in\N}$, which requires all previous lemmata and, as its corollary, we prove our Main Theorem. As the anonymous referee pointed out, the technique of the present paper could be also suitably adapted for proving the renewal theorem for denominators of the \emph{best approximations of the first kind} which are associated with the Farey map.

\section{Preliminaries and Main Result}\label{section: preliminaries and main result}
\subsection{Euclidean Continued Fractions and the Gauss map}
Given $\alpha\in(0,1]\smallsetminus\mathbb{Q}$, its continued fraction expansion is denoted by $$\alpha=\frac{1}{a_1+\frac{1}{a_2+\frac{1}{a_3+\frac{1}{\ddots}}}}=[a_1,a_2,a_3,\ldots,a_n,\ldots],$$
where $a_n\in\N=\{1,2,3,\ldots\}$. The convergents of $\alpha$ are denoted by $\left\{P_n/Q_n\right\}_{n\in\N}$ and defined by $\frac{P_n}{Q_n}=[a_1,a_2,\ldots,a_n]$, with $\mathrm{GCD}(P_n,Q_n)=1$. This kind of continued fractions is called ``Euclidean'' because of its connection with the Euclidean algorithm used to determine the $\mathrm{GCD}$ of two positive integers: given $a,b\in\N$, $a<b$, the partial quotients given by the Euclidean algorithm are the entries of the continued fraction expansion of $\alpha=\frac{a}{b}$.

Let $G$ be the Gauss map, i.e. $G:(0,1]\rightarrow(0,1]$, $G(\alpha)=\left\{\frac{1}{\alpha}\right\}=\frac{1}{\alpha}-\left\lfloor\frac{1}{a}\right\rfloor$, where $\{\cdot\}$ and $\lfloor\cdot\rfloor$ denote the fractional and the integer part respectively. The sequence $\{a_n\}_{n\in\N}$ represents a symbolic coding of the orbit $\{G^n(\alpha)\}_{n\in\N}$, namely $a_n=\left\lfloor\frac{1}{G^{n-1}(\alpha)}\right\rfloor$. A point $\alpha\in(0,1]\smallsetminus\mathbb{Q}$ is thus identified with the sequence $\{a_n\}_{n\in\N}\in\N^\N$. 
There is a natural Markov partition of $(0,1]$ for the Gauss map corresponding to this symbolic representation. $G$ has countably many branches; each branch is surjective and defined on $A(k)=\left(\frac{1}{k+1},\frac{1}{k}\right]$, by $G|_{A(k)}(\alpha)=\frac{1}{\alpha}-k$. 
The map $G$ has an invariant measure $\mu_G$ which is absolutely continuous w.r.t. the Lebesgue measure on $(0,1]$ and it is given by the density $\frac{\de\mu_G}{\de\alpha}=\frac{1}{\log 2}\frac{1}{1+\alpha}$. As a general reference, see \cite{Khinchin35}.

\subsection{Continued Fractions with Even Entries and the corresponding map $T$}
We shall consider a modification of the Euclidean continued fraction expansion of $\alpha\in(0,1]\smallsetminus\mathbb{Q}$, namely an expression of the form 
\begin{equation}\nonumber%\label{eq: continued fractions with even entries}
\alpha=\frac{1}{2k_1+\frac{\xi_1}{2k_2+\frac{\xi_2}{2k_3+\frac{\xi_3}{\ddots}}}}=\left[\left[(k_1,\xi_1),(k_2,\xi_2),(k_3,\xi_3),\ldots,(k_n,\xi_n),\ldots\right]\right],
\end{equation}
where $(k_n,\xi_n)\in\N\times\{\pm1\}=:\Omega$. The corresponding convergents of $\alpha$ are denoted by $\{p_n/q_n\}_{n\in\N}$ and defined by $$\frac{p_n}{q_n}=\les(k_1,\xi_1),(k_2,\xi_2),\ldots,(k_n,*)\res=\frac{1}{2k_1+\frac{\xi_1}{2k_2+\frac{\xi_2}{2k_3+\ldots+\frac{\xi_{n-2}}{2k_{n-1}+\frac{\xi_{n-1}}{2k_n}}}}},$$
where $\mathrm{GCD}(p_n,q_n)=1$ and $``*"$ denotes an arbitrary element of $\{\pm1\}$. The algorithm which corresponds to this kind of continued fractions is similar to the Euclidean one, but at each step it is of the form $v=c\,u+\xi\,r$, with $c\in2\N$, $\xi=\pm1$ and $0\leq r<u$. Given $a,b\in\N$, $a<b$, as input, the sequence of pairs $\{(c_n,\xi_n)\}_{n\in\N}$ generated by the algorithm corresponds to the sequence $\{(k_n,\xi_n)\}_{n\in\N}$ relative to expansion of $\alpha=\frac{a}{b}$, with $k_n=\frac{c_n}{2}$. In this context, the role of the Gauss map is played by the map $T$, defined as follows.
Let us consider the partition of the interval $(0,1]$ given by $\{B(k,\xi)\}_{(k,\xi)\in\Omega}$, where 
\begin{equation}\nonumber%\label{eq: partition Bk+/-}
B(k,-1)=\left(\frac{1}{2k},\frac{1}{2k-1}\right],\hspace{.6cm}B(k,+1)=\left(\frac{1}{2k+1},\frac{1}{2k}\right],\hspace{.5cm}k\in\mathbb{N}
\end{equation}
and let $T$ be the map on $(0,1]$
given by
\begin{equation}\nonumber
T(x)=\xi\cdot\left(\frac{1}{x}-2k\right),\hspace{1cm}x\in B(k,\xi).
\end{equation}
Notice that for $x\in B(k,\xi)$ we have the identity $x=\frac{1}{2k+\xi\, T(x)}$. Therefore  
\begin{equation}\nonumber
x=\les(k_1,\xi_1), (k_2,\xi_2), (k_3,\xi_3), \ldots\res\hspace{.5cm}\mbox{implies}\hspace{.5cm}T^n(x)=\les(k_{n+1},\xi_{n+1}),(k_{n+2},\xi_{n+2}),\ldots\res,
\end{equation}
i.e $T$ acts as a shift over the space $\Omega^\N$.
Continued fraction expansions with even partial quotients and their connection with the map $T$ were initially discussed by F. Schweiger \cite{Schweiger82, Schweiger84} and a deep insight was provided by C. Kraaikamp and A. Lopes \cite{Kraaikamp-Lopes96}, in relation with closed geodesics for the theta group (the subgroup of $\mathrm{SL}(2,\Z)$ generated by $z\mapsto-\frac{1}{z}$ and $z\mapsto z+2$). A detailed analysis of the Euclidean-like algorithm associated to ECF-expansions (and several other expansions) is presented by B. Vall\'ee \cite{Vallee06}.
 
Our interest for such continued fraction expansions and the map $T$ comes from the study of quadratic trigonometric sums of the form $\sum_{n=0}^{N-1}\exp(\pi\,i\,\alpha\, n^2)$, whose renormalization properties are described (see e.g. \cite{Berry-Goldberg-1988, Coutsias-Kazarinoff-1998, Fedotov-Klopp}) by the transformation on $[-1,1]\smallsetminus\{0\}$ defined by $x\mapsto-\frac{1}{x}\,(\mathrm{mod}\,2)$. The restriction of this map to $(0,1]$ is, in modulus, equal to the map $T$.
We shall not dwell on the renormalization of such sums, which will be the subject of our future work.
\begin{figure}
\begin{center}
\includegraphics[width=6cm, angle=0]{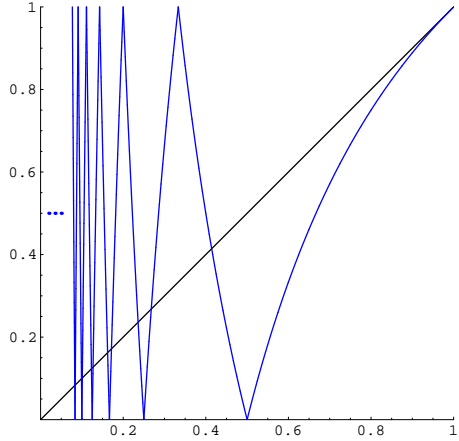}
\caption{\small{Graph of $T$.}} \label{fig: graph of T}
\end{center}
\end{figure}

In comparison with the Gauss map $G$, the map $T$ has a different qualitative behavior. Indeed, it is not uniformly expanding since it has an indifferent fixed point\footnote{A fixed point $x_0$ is called \emph{indifferent} if the map has derivative equal to 1 in modulus at $x_0$. This feature can be found in several maps, which are called \emph{intermittent}. A celebrated example is the Farey map \cite{Feigenbaum-88}, which has an indifferent fixed point at 0. For an overview on intermittent systems see, e.g., \cite{Schuster}.} 
at $x=1$.  Because of this, the map T does not have a finite invariant measure absolutely continuous w.r.t. the
Lebesgue measure. Indeed the expectation of the time spent by a random trajectory in a small neighborhood of $x=1$ is infinite and therefore any invariant measure should give infinite measure to any neighborhood of $x=1$. The following Theorem was proven by F. Schweiger
\cite{Schweiger82}:
\begin{theorem}
The map $T:(0,1]\rightarrow(0,1]$ has a
$\sigma$-finite invariant measure $\nu$ with infinite mass. $\nu$
has density $h(\alpha)=\frac{1}{\alpha+1}-\frac{1}{\alpha-1}$. Moreover $T$ is ergodic, i.e. if $A$ is a $T$-invariant set, then either $\nu(A)=0$ or $\nu(A^c)=0$.
\end{theorem}

%%%%%%% remark on conversion from Euclidean expansions to ECF
\begin{remark}\label{remark: conversion}
It is easy to convert Euclidean continued fraction expansions into ECF-expansions, using the following identity:
\begin{equation}\label{eq: conversion}
a_1+\frac{1}{a_2+\frac{1}{a_3+\gamma}}=(a_1+1)-\frac{1}{2-\frac{1}{2-\ldots-\frac{1}{2-\frac{1}{(a_3+1)+\gamma}}}},
\end{equation}
where ``2'' appears $a_2-1$ times in the right hand side.
Given $\alpha=[a_1,a_2,a_3,\ldots]$,
we can recursively apply the previous identity, moving from the left to the right, to each triplet $(a_{j-1},a_j,a_{j+1})$ such that $a_{j-1}$ is odd and obtain a new sequence $\{(c_n,\xi_n)\}_{n\in\N}$ with $c_n\in2\N$ and $\xi_n\in\{\pm1\}$. Setting $k_n=\frac{c_n}{2}$, we get the sequence $\{(k_n,\xi_n)\}_{n\in\N}$ such that $\alpha=\les(k_1,\xi_1),(k_2,\xi_2),(k_3,\xi_3),\ldots\res$. 
It is easy to see that $1=[[(1,-1),(1,-1),\ldots]]$ and in particular (\ref{eq: conversion}) shows that the ECF-expansion of any rational number is either finite or eventually periodic with $\overline{(1,-1)}$-tail. In our discussion we shall deal only with irrational $\alpha$ for which the ECF-expansion is infinite with no $\overline{(1,-1)}$-tail. Let us denote the set of such sequences with $\dot\Omega^\N$.
\end{remark}

\subsection{The Jump Transformation $R$}
Let $\alpha\in\X=\dot\Omega^\mathbb{N}$, which is endowed with the Borel $\sigma$-algebra $\frak B$. We use the notion of jump
transformation, due to Schweiger (see \cite{Schweiger_ET_Fibred_Systems}, chapter 19), to construct an uniformly expanding map
$R:\X\longrightarrow\X$. Define
\begin{eqnarray}
\tau(\alpha)&:=&\min\left\{j\geq 0\hspace{.2cm}\mbox{s.t.}\hspace{.2cm}
T^j(\alpha)\in
B(1,-1)^c=\left(0,\frac{1}{2}\right]\right\}\hspace{.3cm}\mbox{and}\nonumber\\
R(\alpha)&:=&T^{\tau(\alpha)+1}(\alpha).\nonumber
\end{eqnarray}
$R$ is said to be the \emph{jump transformation\footnote{Some
authors refer to this map as the \emph{induced map} (or the \emph{first passage map}) \emph{w.r.t $\left(0,\frac{1}{2}\right]$}.} associated to $T$ w.r.t. $\left(0,\frac{1}{2}\right]$.}
\begin{figure}[!ht]
\begin{center}
\hspace{0cm}
\parbox{6.0cm}{\includegraphics[width=6.1cm]{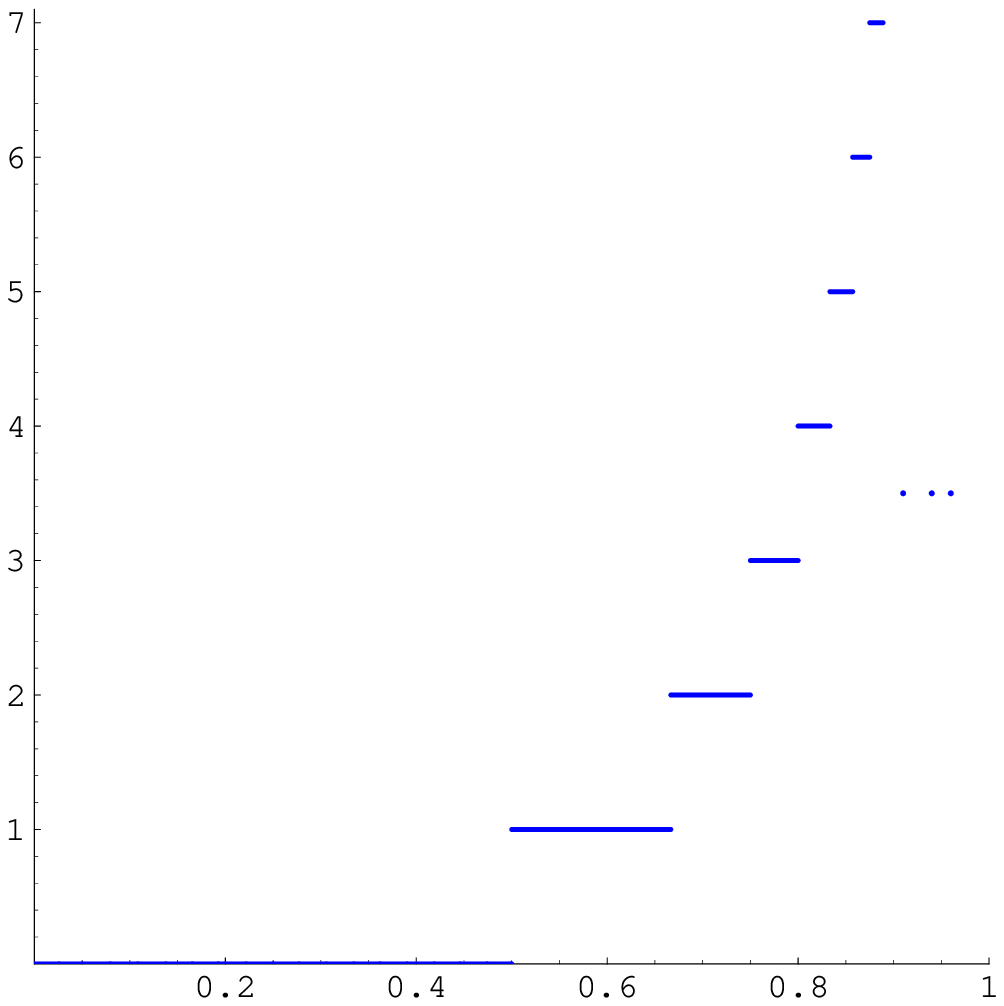}\caption{\small{Graph of $\tau$.}}\label{graph_of_tau}}
\hspace{1cm}
\parbox{6.0cm}{\includegraphics[width=6.1cm]{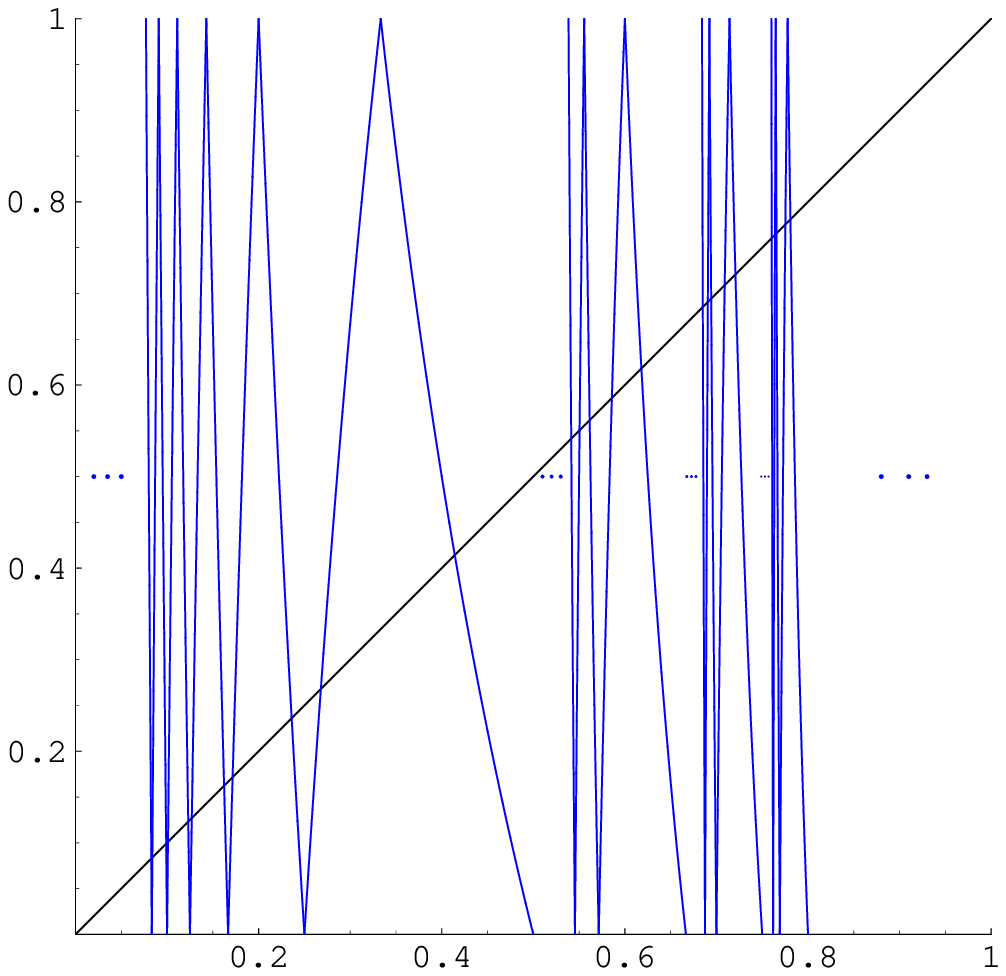}\caption{\small{Graph of $R$.}}\label{graph_of_R}}
\end{center}
\end{figure}
It can be checked that $R$ is uniformly expanding and has bounded distortion, more precisely $\inf_{\alpha\in\X}\left|R'(\alpha)\right|\geq4$ and $\sup_{\alpha\in\X}\left|{R''(\alpha)}/{(R'(\alpha))^2}\right|\leq2$. Therefore, since its branches are surjective, we can deduce by the ``Folklore Theorem'' of Adler \cite{Adler75, Bowen1979} that it has an invariant probability
measure $\mu$ which is absolutely continuous w.r.t. the Lebesgue
measure. Its density $f=\frac{\de\mu}{\de \alpha}$, can be computed explicitly (e.g. using transfer operator
identities, \cite{Vallee06}):
\begin{equation}\nonumber
f(\alpha)=\frac{1}{\log 3}\left(\frac{1}{3-\alpha}+\frac{1}{1+\alpha}\right).
\end{equation}
Let us construct the symbolic representation of the map $R$.
Denote $\overline\omega=(1,-1)\in\Omega$ and
$\Omega^*:=\Omega\smallsetminus\{\overline\omega\}$. Given
$\alpha=\les\omega_1,\omega_2,\omega_3,\ldots\res\in\dot\Omega^\N$ we have that
$\tau=\tau(\alpha)=\min\{j\geq0\hspace{.2cm}\mbox{s.t.}\hspace{.2cm}\omega_{j+1}\neq\overline\omega\}$
and
$R(\alpha)=\les\omega_{\tau+2},\omega_{\tau+3},\omega_{\tau+4},\ldots\res$.
Equivalently,
\begin{eqnarray}
&\les\omega_1,\omega_2,\omega_3,\ldots\res\stackrel{R}{\longmapsto}\les\omega_2,\omega_3,\omega_4,\ldots\res\hspace{.3cm}\mbox{if
$\omega_1\in\Omega^*$;}\nonumber\\
&[[\,\underbrace{\overline\omega,\ldots,\overline\omega}_{\mbox{\tiny{$k$
times}}}\,,\omega_{k+1},\omega_{k+2},\omega_{k+3},\ldots]]\stackrel{R}{\longmapsto}\les\omega_{k+2},\omega_{k+3},\ldots\res\hspace{.3cm}\mbox{if
$\omega_{k+1}\in\Omega^*$}.\nonumber
\end{eqnarray}
Let us set $\Sigma=\N_0\times\Omega^*$, $\N_0=\N\cup\{0\}$,
and denote by $(h,\omega)\in\Sigma$ the word
$(\,\underbrace{\overline\omega,\ldots\overline\omega}_{\mbox{\tiny{$h$
times}}},\omega)$ of length $h+1$ where $\omega\in\Omega^*$. 
In this way we
code each element of
$\dot\Omega^\mathbb{N}$ 
by an element of $\Sigma^\N$ and this coding is clearly invertible. After we identify $\X$ with
$\Sigma^\N$, the map $R:\X\rightarrow\X$ 
becomes a shift on the space $\Sigma^\N$.

For brevity, we denote $m^\pm=0\cdot m^\pm=(0,(m,\pm 1))$ and $h\cdot m^\pm=(h,(m,\pm 1))$. 
By construction, $h\cdot 1^-$ is not allowed for any $h\in\N_0$. 
In the following, we will use both the codings $\X=\dot\Omega^\N$ and $\X=\Sigma^\N$, denoting the elements of $\Omega$ by $\omega$ and the ones of $\Sigma$ by $\sigma$.
\begin{ese}
\begin{eqnarray}
\alpha&=&\frac{1}{6+\frac{1}{14-\frac{1}{2-\frac{1}{2-\frac{1}{2-\frac{1}{4-\frac{1}{10+\frac{1}{\ddots}}}}}}}}=\nonumber\\
&=&\left((3,+1),(7,-1),(1,-1),(1,-1),(1,-1),(2,-1),(5,+1),\ldots\right)\in\dot\Omega^\mathbb{N}\nonumber\\
&=&\left((0,(3,+1)),(0,(7,-1)),(3,(2,-1)),(0,(5,+1)),\ldots\right)\in\Sigma^\N\nonumber\\ 
&=&(0\cdot3^+, 0\cdot7^-,3\cdot2^-,0\cdot5^+,\ldots)=(3^+,7^-,3\cdot 2^-,5^+,\ldots).\nonumber
\end{eqnarray}
\end{ese}

\subsection{Natural Extension of $R$}
The notion of \emph{natural extension} of was introduced by V. Rokhlin \cite{Rokhlin-1961} and since then it became a powerful tool in the study of metric invariants and statistical properties of endomorphisms of measure spaces. Since $R$ is a shift over $\Sigma^\N$, it is easy to construct its natural extension $\hat R$, which will act on the space $D(\hat
R):=\Sigma^\Z$: 
\begin{eqnarray}
(\ldots,\sigma_{-2},\sigma_{-1},\sigma_{0};\sigma_{1},\sigma_{2},\ldots)\stackrel{\hat
R}{\longmapsto}(\ldots,\sigma'_{-2},\sigma'_{-1},\sigma'_{0};\sigma'_{1},\sigma'_{2},\ldots)\nonumber
\end{eqnarray}
where $\sigma'_i=\sigma_{i+1}$. The map $\hat R$ is
clearly invertible and we define the $\sigma$-algebra $\hat{\frak
B}$ on $D(\hat R)$ as the smallest $\sigma$-algebra containing the
preimages $T^{-1}(C)$,
$C\in\frak B$.\\
Given any $\hat\omega\in D(\hat R)$,
$\hat\omega^-=(\sigma_0,\sigma_{-1},\sigma_{-2},\ldots)$ 
and $\hat\omega^+=(\sigma_1,\sigma_2,\sigma_3,\ldots)$ 
will denote the two components of $\hat\omega$, with 
$\sigma_i\in\Sigma$, $i\in\mathbb{Z}$.
Moreover, $\hat R$ has an invariant measure $\hat \mu$ which is
obtained by setting
\begin{eqnarray}
\hspace{.2cm}\hat\mu\left(\left\{\hat\omega\in D(\hat R):\:\sigma_{i_1}\in
C_1,\ldots,\sigma_{i_r}\in C_r\right\}\right)
:=\mu\left(\left\{\hat\omega^+\in
\Sigma^\N:\:\sigma_{i_1+n}\in
C_1,\ldots,\sigma_{i_r+n}\in C_r\right\}\right),\nonumber
\end{eqnarray}
for any $r\in\N$, $i_1,\ldots,i_r\in\mathbb{Z}$ and
$C_1,\ldots,C_r\in\frak B$, where $n\geq 0$ is chosen such that
$i_k+n>0$ for $k=1,\ldots,r$.
\begin{remark}
Coding $(\ldots,\sigma_{-1},\sigma_0;\sigma_1,\sigma_2,\ldots)$ into $(\ldots,\omega_{-1},\omega_{0};\omega_1,\omega_2,\ldots)$ we get  
$\alpha=\hat\omega^+=[[\omega_1,\omega_2,\omega_3\ldots]]\in\X$. On the other hand, we can 
identify $\hat\omega^-$ with a point in $\left(-\frac{1}{3},1\right]\smallsetminus\mathbb{Q}$ by setting $$\hat\omega^-=(\omega_0,\omega_{-1},\omega_{-2},\ldots)=\les(0,\xi_0);(k_0,\xi_{-1}),(k_{-1},\xi_{-2}),\ldots\res=\frac{\xi_0}{2k_0+\frac{\xi_{-1}}{2k_{-1}+\frac{\xi_{-2}}{\ddots}}}.$$
Therefore $D(\hat R)$ is identified with the \virg{rectangle} 
$(0,1]\times\left(-\frac{1}{3},1\right]\smallsetminus\mathbb{Q}^2$.
\end{remark}
The natural extension of the map $T$, denoted by $\hat T$, is defined analogously as an invertible shift over $\Omega^\mathbb{Z}$, the space of bi-sided sequences over the alphabet $\Omega$. Schweiger \cite{Schweiger84} proved that $\hat T$ has an absolutely continuous invariant measure over $(0,1]\times(-1,1]\smallsetminus\mathbb{Q}^2$ which is $\sigma$-finite with infinite mass.

\subsection{$R$-Convergents} 
For $\hat\omega\in D(\hat R)$ 
we set $p_n(\hat\omega)=p_n(\hat\omega^+)$, $q_n(\hat\omega)=q_n(\hat\omega^+)$ and we call $\frac{p_n}{q_n}$ \emph{the $n$-th $T$-convergent} of $\hat\omega$. We are interested in a particular subsequence of $\left\{p_n/q_n\right\}_{n\in\N}$, corresponding to the map $R$, defined as follows.

Set $\theta_0=1$ and
$\theta_i=\theta_i(\hat\omega)=\theta_i(\hat\omega^+)=1+\tau(R^{i-1}(\hat\omega^+))$ for
$i\geq 1$ and define
$\nu_n=\nu_n(\hat\omega)=\nu_n(\hat\omega^+)=\sum_{k=0}^{n}\theta_i$, $n\geq0$. 
The sequence $\{\nu_n\}_n$ gives us the index in $\hat\omega$ of the first $\Omega$-coordinate of $\hat R^n(\hat\omega)$, i.e. $(\hat R^n(\ldots,\omega_0;\omega_1,\omega_2,\ldots))^+=(\omega_{\nu_n},\omega_{\nu_n+1},\ldots)$, $\omega_i\in\Omega$. Furthermore, we have the useful recurrent relation: 
\begin{equation}\nonumber
\nu_n(\hat\omega)=\nu_{n-1}(\hat\omega)+\tau(R^{n-1}(\hat\omega^+))+1.
\end{equation} 
The fraction $\frac{p_{\nu_n}}{q_{\nu_n}}$ is called \emph{the $n$-th $R$-convergent} of $\hat\omega$. We are mainly concerned about the sequence of denominators of the $R$-convergents of $\hat\omega$, i.e. $\hat q_n=\hat q_n(\hat\omega)=\hat q_n(\hat\omega^+):=q_{\nu_n}$, $n\in\N$. The $0$-th denominator is defined in a different way, namely $\hat q_0:=1$. Notice that $\nu_0(\hat\omega)=1$ for each $\hat\omega\in D(\hat R)$ and $1=\hat q_0\neq q_{\nu_0}=2k_1\geq2$.
It is possible to define $\nu_{-n}=\nu_{-n}(\hat\omega)=\nu_{-n}(\hat\omega^-)$ for $n\in\N$ as well, namely $\nu_{-n}=-\sum_{k=1}^{n}\tau\big((\hat R^{-k}(\hat\omega))^+\big)-n+1$.

\begin{ese}
Let $\alpha=\hat\omega^+=\pi-3$.
\begin{eqnarray}
\alpha=\hat\omega^+&\hspace{-.2cm}=\hspace{-.2cm}&(4^-,\underbrace{1^-,\ldots,1^-}_{14\:\mbox{\tiny times}},1^+,146^+,1^-,1^+,1^+,1^-,1^-,1^-,1^+,7^+,2^+,\ldots)=\nonumber\\
&\hspace{-.2cm}=\hspace{-.2cm}&(4^-,14\cdot
1^+,146^+,1\cdot1^+,1^+,3\cdot1^+,7^+,2^+,\ldots),\nonumber
\end{eqnarray}
We have
$\{\theta_k\}_{k=0}^8=\{1,1,15,1,2,1,4,1,1\}$
and hence
$\{\nu_n\}_{n=0}^8=\{1,2,17,18,20,21,25,26,27\}$.
%We have
%\begin{eqnarray}
%&\frac{p_1}{q_1}=\frac{1}{8},\:\:\frac{p_2}{q_2}=\frac{2}{15},\:\:\frac{p_{17}}{q_{17}}=\frac{4687}{33102},\:\:\frac{p_{18}}{q_{18}}=\frac{9390}{66317},\:\:\frac{p_{20}}{q_{20}}=\frac{37576}{265381},\:\:\frac{p_{21}}{q_{21}}=\frac{89245}{630294},\nonumber\\
%&\frac{p_{25}}{q_{25}}=\frac{3612111}{25510582},\:\:\frac{p_{26}}{q_{26}}=\frac{14692696}{103767361},\:\:\frac{p_{27}}{q_{27}}=\frac{25773281}{182024140}\nonumber
%\end{eqnarray}
%and $\{\hat q_n\}_{n=0}^8=\{
%1,\,15,\,33102,\,66317,\,265381,\,630294,\,25510582,\,103767361,\,182024140\}$.
% \small{\begin{eqnarray} \{\hat q_n\}_{n=0}^8=\{
%1,\,15,\,33102,\,66317,\,265381,\,630294,\,25510582,\,103767361,\,182024140\}.\nonumber
%\end{eqnarray}}
\end{ese}

\subsection{Main Result: Renewal-type Theorems for $\{q_n\}$ and $\{\hat q_n\}$}
Given $L>0$, consider the smallest index $n$ for which $q_n$ exceeds $L$, namely
\begin{equation}\nonumber%\label{def: n_L}
n_L=n_L(\hat\omega)=n_L(\hat\omega^+):=\min\left\{n\in\N:\:q_n>L\right\},
\end{equation}
which is referred to as \emph{renewal time} or \emph{waiting time}.
The following Theorem is the main result of this paper.
\begin{maintheorem}%\label{teo: main theorem}
Fix $N_1,N_2\in\N$. The ratio $\frac{q_{n_L}}{L}$ and the entries $\omega_{n_L+j}$ for $-N_1<j\leq N_2$ have a joint limiting probability distribution, as $L\rightarrow\infty$, with respect to the measure $\mu$.
In other words: for each $N_1,N_2\in\N$ there exists a probability measure $\mathrm{P}_{N_1,N_2}$ on $(1,+\infty)\times\Omega^{N_1+N_2}$ such that for all $a,b>1$, $d_j\in\Omega$, $-N_1<j\leq N_2$,
\begin{eqnarray}
\mu\left(\left\{\alpha:\:a<\frac{q_{n_L}}{L}<b,\:\omega_{n_L+j}=d_j,\:-N_1< j\leq N_2\right\}\right)\stackrel{L\rightarrow\infty}{-\hspace{-.1cm}-\hspace{-.2cm}\longrightarrow}\label{eq: main theorem}\\
\mathrm{P}_{N_1,N_2}\big((a,b)\times\{d_{-N_1+1}\}\times\cdots\times\{d_0\}\times\cdots\times\{d_{N_2}\}\big).\nonumber
\end{eqnarray}
\end{maintheorem}
A similar statement was proven by Ya. Sinai and C. Ulcigrai \cite{Sinai-Ulcigrai07},
for the Gauss map and the denominators $\{Q_n\}_{n\in\N}$. They used a special flow and its mixing property to prove the existence of the limiting distribution. Their strategy cannot be applied directly to the sequence $\{q_n\}$ because the map $T$ does not have a finite invariant measure. However, it can be applied to the subsequence $\{\hat q_n\}\subset\{q_n\}$ and 
the Gauss-like map $R$. Thus we first prove the renewal-type Theorem \ref{theorem: renewal for hat q} (see below) for the subsequence $\{\hat q_n\}$ and derive our Main Theorem from it.

Let us define the renewal time for the sequence $\{\hat q_n\}$: given $L>0$, set 
\begin{equation}\nonumber%\label{def: hat n_L}
\hat n_L=\hat n_L(\hat\omega)=\hat n_L(\hat\omega^+):=\min\left\{n\in\N:\:\hat q_n>L\right\}.
\end{equation}
\begin{theorem}\label{theorem: renewal for hat q}%\label{teo: main theorem}
For each $N_1,N_2\in\N$ there exists a probability measure $\mathrm{P}_{N_1,N_2}'$ on $(1,+\infty)\times\Sigma^{N_1+N_2}$ such that for all $a,b>1$, $c_j\in\Sigma$, $-N_1<j\leq N_2$,
\begin{eqnarray}
\mu\left(\left\{\alpha:\:a<\frac{\hat q_{\hat n_L}}{L}<b,\:\sigma_{\hat n_L+j}=c_j,\:-N_1< j\leq N_2\right\}\right)\stackrel{L\rightarrow\infty}{-\hspace{-.1cm}-\hspace{-.2cm}\longrightarrow}\label{eq: main theorem}\\
\mathrm{P}_{N_1,N_2}'\big((a,b)\times\{c_{-N_1+1}\}\times\cdots\times\{c_0\}\times\cdots\times\{c_{N_2}\}\big).\nonumber
\end{eqnarray}
\end{theorem}
The proofs of Theorem \ref{theorem: renewal for hat q} and Main Theorem are given in Section \ref{section: existence of limiting distribution}. 

\subsection{Cylinders}
For $c_i\in\Sigma$, 
$i=1,\ldots,n$, define the cylinder of
length $n$
$$\mathcal{C}[c_1,\ldots,c_n]=\left\{\hat\omega^+=\{\sigma_j\}_{j\in\N}\in\Sigma^\N:\:\sigma_i=c_i,\:1\leq i\leq n\right\}$$
and denote by $\frak C_n^{\,+}$ the set of all cylinders of length
$n$. If $\pi:D(\hat R)=\Sigma^\mathbb{Z}\rightarrow\Sigma^\N=\X$ is
the natural projection and $\mathcal C\in\frak C_n^{\,+}$, then we
shall denote by $\hat{\mathcal{C}}$ the set
$\pi^{-1}\mathcal{C}\subseteq D(\hat R)$. More generally, given
$c_i\in\Sigma$, $-n_1\leq i\leq n_2$, $n_{1,2}\in\N$, set 
$$\mathcal{C}[c_{-n_1},\ldots,c_0;c_1,\ldots,c_{n_2}]=\left\{
\{\sigma_j\}_{j\in\mathbb{Z}}\in D(\hat R):\:\sigma_i=c_i,\:-n_1\leq i\leq n_2\right\}$$
and denote by $\frak C_{n_1,n_2}$ the set of all such bi-sided
cylinders.\\ 
\begin{remark}\label{remark: measure of cylinders}
Notice that $\mathcal{C}[m^\pm]=\mathcal{C}[0\cdot m^\pm]=B(m,\pm 1)$. Moreover, after writing explicitly all cylinders of length one 
and integrating the density $f$, for all $h\cdot m^\pm\in\Sigma$, we get
\begin{eqnarray}
\mu\left(\mathcal{C}[h\cdot m^\pm]\right)&\leq&\frac{3}{\log
3\,(4h^2+8h+3)}\frac{1}{m^2}.\nonumber 
\end{eqnarray}
In particular, the measures of our cylinders of length one are $\mathcal{O}\left(\frac{1}{m^2}\right)$ as $m\rightarrow\infty$, where the constants implied by the $\mathcal{O}$-notation
depend on $h$ as above. 
\end{remark}

\subsection{Special Flows}\label{section: special flows}
Consider a probability space $(D,\mathcal{B},\mu)$, an invertible
$\mu$-preserving map $F:D\rightarrow D$ (the \virg{base}
transformation) and a positive function
$\varphi:D\rightarrow\mathbb{R}^+$ (the \virg{roof} function) such
that $I=\int_D\varphi(\omega)\,\de\mu(\omega)<\infty$. Define
$D_\Phi=\{(x,y)\in D\times\mathbb{R}:\:0\leq
y<\varphi(x)\}\subseteq D\times\mathbb{R}$ and set $\mu_\Phi$ as
the normalized measure obtained by restriction of the
product measure $I^{-1}\,\mu\times\lambda$ to $D_\Phi$, where $\lambda$ is the
Lebesgue measure on $\mathbb{R}$. The \emph{special flow
$\{\Phi_t\}_{t\in\mathbb{R}}$ built over $F$ under the roof
function $\varphi$} is the one-parameter group of
$\mu_\Phi$-preserving transformations on $D_\Phi$ whose action is
defined as follows (see e.g. \cite{CFS}):
\begin{equation}\nonumber
\left\{%
\begin{array}{ll}
    \Phi_t(x,y)=(x,y+t), & \hbox{if $0\leq y+t<\varphi(x)$;} \\
    \Phi_{\varphi(x)}(x,0)=(F(x),0). & \hbox{ } \\
\end{array}%
\right.
\end{equation}
The flow moves a point $(x,y)\in D_\Phi$ vertically upward with unit
speed to the ``roof'' point $(x,\varphi(x))$. After that, the point jumps
to the ``base'' point $(F(x),0)$  and continues moving vertically until the next jump and so on.
We shall denote by
$$S_0(\varphi,F)(x):=0,\hspace{.8cm}S_r(\varphi,F)(x):=\sum_{i=0}^{r-1}\varphi\left(F^i(x)\right),\hspace{.8cm} x\in D,\,r\in\N$$
the \emph{$r$-th (non-normalized) Birkhoff sum of $\varphi$ along
the trajectory of $x$ under $F$}. Given $x\in D$ and
$t\in\mathbb{R}^+$ we define $r(x,t)\in\N$,
$$r(x,t):=\min\{r\in\N\,:\:S_r(\varphi,F)(x)>t\}.$$
The non-negative integer $r(x,t)-1$ is the number of discrete
iterations of $F$ which the point $(x,0)\in D_\Phi$ undergoes
before time $t$. The flow $\Phi_t$ defined above acts therefore
for $t>0$ as
\begin{equation}\label{eq: action special flow Phi_t}
\Phi_t(x,0)=\left(F^{r(x,t)-1}(x),t-S_{r(x,t)-1}(\varphi,T)(x)\right),
\end{equation}
while for $t<0$ the action of the flow is defined using the
inverse map.

\section{Recurrence Relations and Denominator Estimates}\label{section: recurrence relations and denominator estimates}
Consider the Euclidean continued fraction expansion $[a_1,a_2,a_3,\ldots]$ of $\alpha\in\X$. 
We have the recurrence formulae
\begin{equation}\nonumber
P_{n}=a_{n}\,P_{n-1}+P_{n-2},\hspace{.8cm}
Q_{n}=a_{n}\,Q_{n-1}+Q_{n-2},\hspace{.3cm}n\in\N,
\end{equation}
with $Q_{-1}=P_0=0,\:\:P_{-1}=Q_0=1$ and two well-known estimates
on the quality of the approximation and on the growth of the
denominators, \cite{Khinchin35}:
\begin{equation}\label{eq: estimates for Euclidean CF}
\left|x-\frac{P_n}{Q_n}\right|\leq\frac{1}{Q_n^2},\hspace{.9cm}Q_n\geq2^{\frac{n-1}{2}},\hspace{.6cm}n\in\N
\end{equation}
Similar recurrence formulae are valid for ECF-expansions (see
\cite{Kraaikamp-Lopes96}): for $\alpha=\les(k_1,\xi_1),(k_2,\xi_2),\ldots\res$ we have
\begin{equation}\label{eq: recursion p n and q n}
p_{n}=2k_{n}\,p_{n-1}+\xi_{n-1}\,p_{n-2},\hspace{.8cm}
q_{n}=2k_{n}\,q_{n-1}+\xi_{n-1}\,q_{n-2},\hspace{.3cm}n\in\N,
\end{equation}
with $q_{-1}=p_{0}=0,\:\:p_{-1}=q_0=\xi_0={1}$. However, estimates
analogous to (\ref{eq: estimates for Euclidean CF}) are not
available for ECF-expansions. We have indeed the weaker estimates
(see \cite{Kraaikamp-Lopes96} and \cite{Schweiger82})
\begin{equation}\label{eq: estimates for ECF}
\left|x-\frac{p_n}{q_n}\right|\leq\frac{1}{q_n},\hspace{.9cm}q_n\geq
n+1,\hspace{.6cm}n\in\N
\end{equation}
and the latter, in particular, is optimal and cannot be improved.
Nevertheless, we claim that the sequence $\{\hat q_n\}_n$ grows \emph{at
least} exponentially fast. Indeed we prove the following Lemma
which provides an estimate similar to (\ref{eq: estimates for
Euclidean CF}).
\begin{lem}[Growth of $R$-denominators, lower bound]
For any $\hat\omega\in D(\hat R)$, the denominators $\hat q_n=\hat
q_n(\hat\omega)=\hat q_n(\hat\omega^+)$ satisfy the estimate
\begin{equation}\label{eq: growth of hat q n}
\hat q_n\geq3^{\frac{n}{3}},\hspace{.6cm}n\in\N.
\end{equation}
\end{lem}
\begin{proof}
Recall %that $\hat q_n=q_{\nu_n}$ and 
that $\nu_{n}=\nu_{n-1}+1+\tau(R^{n-1}(\hat\omega^+))$ for $n\in\N$. 
%Denote for brevity $\tau=\tau(R^{n-1}(\hat\omega^+))$ and
%$j=\nu_{n-1}$. Suppose we know the
%values $q_{j-2}=A$ and $q_{j-1}=B$, with $A<B$.\\
%If $\tau=0$, then $\omega_{j}=(k_j,\xi_j)\in\Omega^*$, while
%$\omega_{j-1}=(k_{j-1},\xi_{j-1})$ and
%$\omega_{j+1}=(k_{j+1},\xi_{j+1})$ are arbitrary elements of
%$\Omega$. Recall that if $\xi_j=+1$, then $k_j\geq1$ and if
%$\xi_j=-1$, then $k_j\geq2$. 
Using (\ref{eq: recursion p n and q n})
%and the fact that $k_{j+1}\geq 1$ we get
one can see that $\hat q_n=q_{\nu_n}\geq3\, q_{\nu_{n-1}-2}$.
%\begin{eqnarray}
%q_j&=&2k_j\,q_{j-1}+\xi_{j-1}\,q_{j-2}=2k_jB+\xi_{j-1}A\hspace{.5cm}\mbox{and}\nonumber\\
%q_{j+1}&=&2k_{j+1}\,q_{j}+\xi_{j}\,q_{j-1}=2k_{j+1}\left(2k_jB+\xi_{j-1}A\right)+\xi_{j}B\geq\nonumber\\
%&\geq&(4k_j+\xi_j)B-2A\geq\min\left\{5B-2A,\:7B-2A\right\}\geq3A.\nonumber
%\end{eqnarray}
%If $\tau=l\geq1$, then
%$\omega_{j}=\omega_{j+1}=\ldots=\omega_{j+l-1}=(1,-1)$ and
%$\omega_{j+l}=(k_{j+l},\xi_{j+l})\in\Omega^*$, while
%$\omega_{j-1}=(k_{j-1},\xi_{j-1})$ and
%$\omega_{j+l+1}=(k_{j+l+1},\xi_{j+l+1})$ are arbitrary elements of
%$\Omega$. As before, $\xi_{j+l}=+1$ implies $k_{j+l}\geq1$ and
%$\xi_{j+l}=-1$ implies $k_{j+l}\geq2$. Using (\ref{eq: recursion p n
%and q n}) and the fact that $k_{j+l+1}\geq1$ we get
%\begin{eqnarray}
%q_j&=&2k_j\,q_{j-1}+\xi_{j-1}\,q_{j-2}=2B+\xi_{j-1}A,\nonumber\\
%q_{j+1}&=&2k_{j+1}\,q_{j}+\xi_{j}\,q_{j-1}=3B+\xi_{j-1}A,\nonumber\\
%&&\ldots\nonumber\\
%q_{j+l-1}&=&(l+1)B+\xi_{j-1}A,\nonumber\\
%q_{j+l}&=&
%2k_{j+l}\big((l+1)B+\xi_{j-1}A\big)-(l
%B+\xi_{j-1}A)\hspace{.5cm}\mbox{and}\nonumber\\
%q_{j+l+1}&=&2k_{j+l+1}\Big(2k_{j+l}\big((l+1)B+\xi_{j-1}A\big)-(lB+\xi_{j-1}A)\Big)+\xi_{j+l}\big((l+1)B+\xi_{j-1}A\big)\geq\nonumber\\
%&\geq&4k_{j+l}\big((l+1)B-A\big)-2(lB-A)+\xi_{j+l}\big((l+1)B-A\big)=\nonumber\\
%&=&\big((4k_{j+l}+\xi_{j+l})(l+1)-2l\big)B-(4k_{j+l}+\xi_{j+l}-2)A\geq\nonumber\\
%&\geq&\min\left\{(3l+5)B-3A,\:(5l+7)B-5A\right\}\geq(5l+2)A\geq7A.\nonumber
%\end{eqnarray}
%Thus we have shown that
In particular we get
for any $n\geq4$
\begin{eqnarray}\label{eq: recursive estimate denominator q}
\hat q_n=q_{\nu_n}
\geq 3\,q_{\nu_{n-1}-2}\geq3\,q_{\nu_{n-3}}=3\hat q_{n-3},
\end{eqnarray}
where the second inequality follows from the monotonicity of the
sequence $\{q_n\}$ because $\nu_{n-1}-\nu_{n-3}\geq2$. Now, by the
second inequality of (\ref{eq: estimates for ECF}), we have
$$\hat q_i\geq q_{i+1}\geq i+2
\hspace{.5cm}\mbox{for $i=1,2,3$}$$
and therefore from (\ref{eq: recursive estimate denominator q}) we
get the estimate for $n\in\N$ 
\begin{equation}\nonumber%\label{eq: estimate hat q n, n in N}
\hat
q_n=q_{\nu_n}\geq
\left([n-1]_{3}+3\right)\cdot3^{\left(\left\lceil\frac{n}{3}\right\rceil-1\right)}
\geq3^{\left\lceil\frac{n}{3}\right\rceil}\geq3^{\frac{n}{3}},
\end{equation}
where $[p]_3=p\,(\mathrm{mod}\,3)$ and $\lceil p\rceil:=\min\{m\in\N:\:m\geq p\}$. 
This
concludes the proof of the Lemma. 
\end{proof}
\begin{remark}
Our proof actually gives $\hat q_n\geq\max\left\{n+2,\:3^{\frac{n}{3}}\right\}$, $n\in\N$, which can be replaced by (\ref{eq: growth of hat q n}) for $n\geq6$. However, 
(\ref{eq: growth of hat q n}) will be enough for our purposes.
\end{remark}
%We just proved that the $R$-denominators $\hat q_n$ grow \emph{at least} exponentially fast. 
The following Lemma provides an upper bound for the growth of the $R$-denominators $\hat q_n$, proving that \emph{typically} (i.e. $\mu$-almost surely) they grow \emph{at most} exponentially fast. The proof is analogous to the one given by Khinchin \cite{Khinchin35} for the Euclidean continued fraction expansions.
%As in the theory of Euclidean continued fractions, also the $R$-denominators $\hat q_n$ do not grow faster then 
\begin{lem}[Growth of $R$-denominators, upper bound]\label{lem: growth of hat q n, upper bound}
There exists a constant $C_{1}>0$ such that for $\mu$-almost every $\hat\omega^+\in(0,1)\smallsetminus\mathbb{Q}=\Sigma^{\N}$  the denominators $\hat q_n=\hat q_n(\hat\omega^+)$ satisfy the estimate
\begin{equation}\label{eq: growth of hat q n, upper bound}
\hat q_n\leq e^{C_{1}\, n}
\end{equation}
for all sufficiently large $n$.
%For any $\hat\omega\in D(\hat R)$, the denominators $\hat q_n=\hat
%q_n(\hat\omega)=\hat q_n(\hat\omega^+)$ satisfy the estimate
%\begin{equation}\label{eq: growth of hat q n}
%\hat q_n\geq3^{\frac{n}{3}},\hspace{.6cm}n\in\N.
%\end{equation}
\end{lem}
\begin{proof}
Let $\hat\omega^+=(h_1\cdot m_1^{\pm},h_2\cdot m_2^{\pm},\ldots)\in\Sigma^{\N}$. By the definition of $\nu_n$ we get
\begin{equation}\label{eq: lemma upper bound 1}
\hat q_n= q_{\nu_n}=q_{h_1+\ldots+h_n+n+1}\leq q_{h_1+\ldots+h_{n+1}+n+1}=q_{\nu_{n+1}-1}.
\end{equation}
Now, using (\ref{eq: recursion p n and q n}) one can show that $\frac{1}{2}\,m_n\,(h_n+1)\,q_{\nu_{n-1}-1}\leq q_{\nu_n-1}\leq6\,m_n\,(h_n+1)\,q_{\nu_{n-1}-1}$ and therefore
%\begin{eqnarray}
 $\frac{1}{2^n}\,m_n\,(h_n+1)\cdots m_1\,(h_1+1)\leq q_{\nu_n-1}\leq 6^n\,m_n\,(h_n+1)\cdots m_1\,(h_1+1)
 $. %\label{eq: lemma upper bound 2}
%\end{eqnarray}
Defining $a_{2j-1}:=m_j$ and $a_{2j}:=(h_j+1)$ for $j\in\N$, the previous inequalities become %second inequality of (\ref{eq: lemma upper bound 2}) becomes
\begin{equation}\label{eq: lemma upper bound 2}%\label{eq: lemma upper bound 3}
\frac{1}{2^n}\,\prod_{j=1}^{2n}a_j\leq q_{\nu_n-1}\leq 6^n\,\prod_{j=1}^{2n}a_j.
\end{equation}
Let us show that the product $\prod_{j=1}^{2n}a_j$ is bounded by $e^{A\,n}$ Lebesgue-almost surely for some $A>0$. For $s\geq1$ define $E_n(s):=\left\{\hat\omega^+\in\Sigma^\N:\:\prod_{j=1}^{2n}a_j\geq s\right\}$. This set can be written as union of intervals of the form $J_n=\left\{\hat\omega^+\in\Sigma^\N:\:\big((\hat\omega^+)_j\big)_{j=1}^n=(h_1\cdot m_1^\pm,\ldots,h_n\cdot m_n^\pm)\right\}$ and each of these intervals has length $|J_n|\leq\frac{4}{3}\,\frac{1}{q_{\nu_n-1}^2}$. Thus, by the first inequality of (\ref{eq: lemma upper bound 2}), $|J_n|\leq\frac{4}{3}\,2^{2n}\,\prod_{j=1}^{2n}\frac{1}{a_j^2}$. %, where $|\cdot|$ denotes the Lebesgue measure. 
Reasoning as in \cite{Khinchin35} (\S III.14) we obtain
\begin{eqnarray}
|E_n(s)|<\frac{4}{3}\:2^{2n}\cdot\hspace{-.4cm}\sum_{\tiny{\begin{array}{c}a_1,\ldots,a_{2n}\in\N\\a_1\cdots a_{2n}\geq s\end{array}}}
%{a_1\cdots a_{2n}\geq s}
\prod_{j=1}^{2n}\frac{1}{a_j^2}<\frac{4}{3}\cdot\frac{2^{4n}}{s}\,\sum_{j=0}^{2n-1}\frac{(\log s)^j}{j!}.\nonumber
\end{eqnarray}
In particular for $s=e^{A\,n}$ by Stirling's formula one gets
\begin{eqnarray}
|E_n(e^{A\,n})|<\frac{4}{3}\,e^{n(4\log 2-A)}(2n)\frac{(A\,n)^{2n}}{(2n)!}2^{2n}\leq C_{2}\,\sqrt{n}\,e^{-n(A-2\log A-2\log2-2)},\nonumber
\end{eqnarray}
for some $C_{2}>0$. Choosing $A$ so that $(A-2\log A-2\log2-2)>0$ we give an upper bound for $|E_n(e^{A\,n})|$ by the $n$-th term of a convergent series. Thus we get $\sum_{n=1}^\infty|E_n(e^{A\,n})|<\infty$ and therefore Lebesgue-almost every $\hat\omega^+\in(0,1)$ belongs only to a finite number of $E_n(e^{A\,n})$'s. In other words, for sufficiently large $n$, $\prod_{j=1}^{2n}a_j\leq e^{A\,n}$ Lebesgue-almost surely. Now, by (\ref{eq: lemma upper bound 1}) and (\ref{eq: lemma upper bound 2}) we get
\begin{equation}\nonumber
\hat q_n\leq q_{\nu_{n+1}-1}\leq6^{n+1}\,\prod_{j=1}^{2n+2}a_j\leq6^{n+1}\,e^{A\,(n+1)}\leq e^{C_{1}\,n}
\end{equation}
for some $C_{1}>0$ for Lebesgue-almost every $\hat\omega^+\in\Sigma^\N$ and for all sufficiently large $n$. The assertion of the Lemma follows now from the absolute continuity of $\mu$ w.r.t. the Lebesgue measure on $(0,1]$.
\end{proof}

\section{Reduction to a Special Flow}\label{section: reduction to a special flow}
\subsection{Roof function}
For $\hat\omega=(\ldots,\omega_{-1},\omega_0;\omega_1,\omega_2,\ldots)$, with
$\omega_i=(k_i,\xi_i)\in\Omega$, $i\in\mathbb{Z}$, we define
\begin{eqnarray}
\psi(\hat\omega):=\sum_{i=2}^{\nu_1(\hat\omega)}\log\left(\frac{\xi_i(\hat\omega)}{(\hat
T^i \hat
\omega)^-}\right)=\sum_{i=2}^{\nu_1(\hat\omega)}\log\left(2k_i+\frac{\xi_{i-1}}{2k_{i-1}+\frac{\xi_{i-2}}{2k_{i-2}+\frac{\xi_{i-3}}{\ddots}}}\right)\label{def:
roof function psi}
\end{eqnarray}
as roof function over $\hat R$.
The reason for this definition will be clear from Lemma \ref{lemma: approximation by Birkhoff sums}.
\begin{remark}Recall that $\nu_0(\hat\omega)=1$ for every
$\hat\omega\in D(\hat R)$ and that $\nu_i(\hat
R^j\omega)=\nu_{i+j}(\hat\omega)$ for every 
$i,j\geq0$. Hence we have
\begin{eqnarray}
\psi(\hat R^{j}\hat\omega)=\sum_{i=\nu_0(\hat
R^{j}\hat\omega)+1}^{\nu_1(\hat
R^{j}\hat\omega)}\log\left(\frac{\xi_i(\hat
R^{j}\hat\omega)}{(\hat T^i\hat
R^{j}\hat\omega)^-}\right)=\sum_{i=\nu_{j}(\hat\omega)+1}^{\nu_{j+1}(\hat\omega)}\log\left(\frac{\xi_i(\hat\omega)}{(\hat
T^i\hat\omega)^-}\right).\label{def: roof function
psi(R j-1)}
\end{eqnarray}
\end{remark}
In order to construct the special flow as described in Section \ref{section: special flows}, we have to check that our roof function is integrable. 
\begin{lem}\label{lemma: psi is integrable}
The roof function $\psi:D(\hat R)\rightarrow\mathbb{R}^{+}$ is $\hat\mu$-integrable.
\end{lem}
\begin{proof}
Let us denote $\tau=\tau(\hat\omega)=\tau(\hat\omega^+)$ and recall that $\nu_1(\hat\omega)=\tau+2$. We can write $\psi=\psi_0+\psi_1+\psi_2$, where $\psi_{0,1,2}:D(\hat R)\rightarrow\mathbb{R}_{\geq0}$\,,
\begin{eqnarray}
\psi_0(\hat\omega)&=&\begin{cases}
      \displaystyle{\sum_{i=2}^{\tau}\log\left(\frac{\xi_i(\hat\omega)}{(\hat T^i \hat\omega)^-}\right)}& \text{if $\tau\geq2$}, \\
      \hspace{1cm}0& \text{otherwise}.
\end{cases}\hspace{.3cm}\mbox{and}\nonumber\\
\psi_j(\hat\omega)&=&\begin{cases}
      \displaystyle{\log\left(\frac{\xi_{\tau+j}(\hat\omega)}{(\hat T^{\tau+j} \hat\omega)^-}\right)}& \text{if ($j=1$ and $\tau\geq1$) or $j=2$}, \\
      \hspace{1cm}0& \text{if $j=1$ and $\tau=0$}.
\end{cases}\nonumber 
\end{eqnarray}
Let $-\frac{1}{3}<x<1$ such that $x=\les(0,\xi_0);(k_0,\xi_{-1}),(k_{-1},\xi_{-2}),\ldots\res$.
If $\tau\geq2$, we have, for $i=2,\ldots,\tau$
$$\frac{\xi_i(\hat\omega)}{(\hat T^i \hat\omega)^-}=\frac{(i+1)-i\,x}{i-(i-1)x}\hspace{.4cm}\mbox{and therefore}\hspace{.4cm}\psi_0(\hat\omega)=\log\left(\frac{(\tau+1)-\tau\,x}{2-x}\right)$$
because the sum defining $\psi_0$ is telescopic. Using the fact that $\tau=h$ iff $\hat\omega\in\hat{\mathcal{C}}[h\cdot m^\pm]$, the definition of the measure $\hat\mu$ and the estimates given in Remark \ref{remark: measure of cylinders} we get
\begin{eqnarray}
\int_0^1\psi_0(\hat\omega)\,\mathrm{d}\hat\mu(\hat\omega)&=&\sum_{\tiny{
\left.\begin{array}{c}h\cdot m^{\pm}\in\Sigma\\h\geq2\end{array}\right.}}
\int_{\hat{\mathcal{C}}[h\cdot m^\pm]}
\log\left(\frac{(h+1)-h\,x}{2-x}\right)\mathrm{d}\hat\mu(\hat\omega)\leq\nonumber\\
&\leq&\sum_{h\geq2}\log(2h+1)\hspace{-.1cm}\sum_{m^\pm\in\Omega^*}\mu\left(\mathcal{C}[h\cdot m^\pm]\right)\leq\nonumber\\
&\leq&\frac{6}{\log3}\sum_{h\geq2}\frac{\log(2h+1)}{4h^2+8h+3}\sum_{m\geq1}\frac{1}{m^2}\leq3.\nonumber
\end{eqnarray}
\begin{eqnarray}
\int_0^1\psi_1(\hat\omega)\,\mathrm{d}\hat\mu(\hat\omega)&=&\sum_{\tiny{
\left.\begin{array}{c}h\cdot m^{\pm}\in\Sigma\\h\geq1\end{array}\right.}}
\int_{\hat{\mathcal{C}}[h\cdot m^\pm]}\log\left(2m-\frac{h-(h-1)x}{(h+1)-h\,x}\right)\mathrm{d}\hat\mu(\hat\omega)\leq\nonumber\\
&\leq&\sum_{m^\pm\in\Omega^*}\log(2m)\sum_{h\geq1}\mu\left(\mathcal{C}[h\cdot m^\pm]\right)\leq\nonumber\\
&\leq&\frac{6}{\log3}\sum_{m\geq1}\frac{\log(2m)}{m^2}\sum_{h\geq1}\frac{1}{4h^2+8h+3}\leq2.\nonumber
\end{eqnarray}
Let us estimate the integral of $\psi_2$.
Notice that $k_1(\hat\omega)=m$ for $\hat\omega\in\hat{\mathcal{C}}[m^\pm]$ and  $k_1(\hat\omega)=1$ for $\hat\omega\in\hat{\mathcal{C}}[h\cdot m^\pm]$ with $h\in\N$. Moreover, $B(1,-1)=\left[\frac{1}{2},1\right]$ can be written as the disjoint union of $\mathcal{C}[h\cdot m^\pm]$, $h\in\N$, $m\in\Omega^*$, and its $\mu$-measure is $\mu\left(B(1,-1)\right)=\log\frac{5}{3}$.
Now, using also the $\hat R$-invariance of the measure $\hat\mu$ we get
\begin{eqnarray}
\int_0^1\psi_2(\hat\omega)\,\mathrm{d}\hat\mu(\hat\omega)&\leq&\int_0^1\log(2k_{\nu_1}+1)\,\mathrm{d}\hat\mu(\hat\omega)=\int_0^1\log(2k_1+1)\,\mathrm{d}\hat\mu(\hat\omega)\leq\nonumber\\
&\leq&\frac{6}{\log3}\sum_{m\geq1}\frac{\log(2m+1)}{m^2}+\log3\cdot\mu\left(B(1,-1)\right)\leq15\nonumber
\end{eqnarray}
and this concludes the proof of the Lemma.
\end{proof}
\begin{lem}\label{lemma: delta and M} 
Let $\mathcal{C}\in\frak{C}_{n_1,n_2}$, with $n_1\in\N$ and $n_2\geq2$, be any cylinder. Then there exist $\delta=\delta(\mathcal{C})>0$ and $M=M(\mathcal{C})>0$ such that
$$\inf_{\hat\omega\in\mathcal{C}}\psi(\hat\omega)\geq\delta\hspace{.5cm}\mbox{and}\hspace{.5cm}\sup_{\hat\omega\in\mathcal{C}}\psi(\hat\omega)\leq M.$$
\end{lem}
\begin{proof}
Using the same notations introduced in the proof of Lemma \ref{lemma: psi is integrable}, the statement of Lemma \ref{lemma: delta and M} follows easily from the following elementary estimates for $\psi_{1,2,3}$ on cylinders in $\frak{C}^+_2$.\\
If $\hat\omega\in\mathcal{C}[0\cdot m_1^\pm,0\cdot m_2^\pm]$, 
then $\psi_0(\hat\omega)=\psi_1(\hat\omega)=0$ and 
$0<\log\frac{9}{5}\leq\log\left(\frac{10m_2-1}{5}\right)\leq\psi_2(\hat\omega)\leq\log\left(2m_2+1\right)$.\\
If $\hat\omega\in\mathcal{C}[0\cdot m_1^\pm,h_2\cdot m_2^\pm]$, $h_2\geq1$, then 
$\psi_0(\hat\omega)=\psi_1(\hat\omega)=0$ and 
$0<\log\frac{9}{5}\leq\psi_2(\hat\omega)\leq\log3$.\\
If $\hat\omega\in\mathcal{C}[h_1\cdot m_1^\pm,0\cdot m_2^\pm]$, $h_1\geq1$,  
then $0\leq\psi_0(\hat\omega)\leq\log\left(2h_1+1\right)$, $0\leq\psi_1(\hat\omega)\leq\log\left(2 m_1\right)$ and 
$0<\log\frac{5}{3}\leq\log\left(\frac{6m_2-1}{3}\right)\leq\psi_2(\hat\omega)\leq\log\left(2m_2+1\right)$.\\
If $\hat\omega\in\mathcal{C}[h_1\cdot m_1^\pm,h_2\cdot m_2^\pm]$, $h_{1,2}\geq1$, 
then $0\leq\psi_0(\hat\omega)\leq\log\left(2h_1+1\right)$, $0\leq\psi_1(\hat\omega)\leq\log\left(2 m_1\right)$ and 
$0<\log\frac{5}{3}\leq\psi_2(\hat\omega)\leq\log3$.
\end{proof}

\subsection{Mixing of the Special Flow}
Let us consider the special flow $\{\Phi_t\}_{t\in\mathbb{R}}$ built over $\hat R$ under the roof function $\psi$. Set $I=\int_{\mathcal{X}}\psi\,\mathrm{d}\hat\mu$ (it is finite by Lemma \ref{lemma: psi is integrable}) and let $\tilde\mu=I^{-1}\hat\mu\times\lambda$ be the $\Phi_t$-invariant probability measure on $D_\Phi$. Recall that $\{\Phi_t\}_{t\in\mathbb{R}}$ is said to be \emph{mixing} if, for all Borel subsets $\mathrm A$ and $\mathrm{B}$ of $D_\Phi$, we have $$\lim_{t\rightarrow\infty}\tilde\mu\left(\Phi_{-t}(\mathrm{A})\cap\mathrm{B}\right)=\tilde\mu(\mathrm{A})\,\tilde\mu(\mathrm{B}).$$
\begin{prop}\label{prop: the flow is mixing}
The flow $\{\Phi_t\}_{t\in\mathbb{R}}$ is mixing.
\end{prop}
Proposition \ref{prop: the flow is mixing} is proven in Appendix A. 

\subsection{Approximation by Birkhoff sums}
The following Lemma shows that $\log\hat q_n(\hat\omega)$ can be approximated by a Birkhoff sum of the function $\psi$ along the trajectory of $\hat\omega$ under $\hat R$. Moreover this approximation is uniform in $\hat\omega$ and exponentially accurate as $n\rightarrow\infty$. Define
\begin{eqnarray}
g_n(\hat\omega)&:=&\log \hat q_n(\hat\omega)-S_{n}(\psi,\hat
R)(\hat\omega),\hspace{.6cm}n\geq0.\label{def: gn}
\end{eqnarray}
\begin{lem}\label{lemma: approximation by Birkhoff sums}
There exists a function $g$ on $D(\hat R)$ such that $g_n$ converges to $g$ uniformly in $\hat \omega$ and exponentially fast in $n$, i.e.
\begin{equation}\label{eq: statement lemma birkhoff sum denominators}
\log\hat q_n(\hat\omega)=S_{n}(\psi,\hat R)(\hat\omega)+g(\hat\omega)+\varepsilon_n(\hat\omega),\hspace{.5cm}\sup_{\hat\omega\in D(\hat R)}\left|\varepsilon_n(\hat\omega)\right|=\mathcal{O}(3^{-\frac{n}{3}}).
\end{equation}
More precisely: for $n\geq4$ we have $\sup_{\hat\omega\in D(\hat R)}\left|\varepsilon_n(\hat\omega)\right|\leq C_3\,3^{-\frac{n}{3}}$, for some $C_3>0$.
\end{lem}
\begin{proof}
Let $r_n=\frac{q_n}{q_{n-1}}$, $n\in\N$. Since $q_0=1$ we have
$q_n=\prod_{i=1}^n r_i$. From (\ref{eq: recursion p n and q n}) we
get
\begin{eqnarray}
\hspace{.1cm}r_{i}=2k_{i}+\frac{\xi_{i-1}}{r_{i-1}}=2k_{i}+\frac{\xi_{i-1}}{2k_{i-1}+\frac{\xi_{i-2}}{2k_{i-2}+\ldots+\frac{\xi_2}{2k_2+\frac{\xi_1}{2k_1}}}}=\les(k_{i},\xi_{i-1}),(k_{i-1},\xi_{i-2}),\ldots,(k_2,\xi_1),(k_1,*)\res^{-1}\nonumber
\end{eqnarray}
and hence
\begin{equation}\label{eq: log q n}
\log \hat q_n=-\sum_{i=1}^{\nu_n(\hat\omega)}\log
\les(k_{i},\xi_{i-1}),(k_{i-1},\xi_{i-2}),\ldots,(k_2,\xi_1),(k_1,*)\res.
\end{equation}
From (\ref{def: roof function psi(R j-1)}) we get
\begin{eqnarray}
\hspace{-.2cm}S_{n}(\psi,\hat R)(\hat\omega)&\hspace{-.25cm}=&\hspace{-.25cm}\sum_{j=0}^{n-1}\psi(\hat
R^{j}\hat\omega)=
\sum_{j=0}^{n-1}\left(\sum_{i=\nu_j(\hat\omega)+1}^{\nu_{j+1}(\hat\omega)}\log\left(\frac{\xi_i(\hat\omega)}{(\hat
T^i\hat\omega)^-}\right)\right)\hspace{-.1cm}.\label{eq: Sn(varphi,T)(hat
omega)}
\end{eqnarray}
From the definition (\ref{def: gn}) and (\ref{eq: log q
n}-\ref{eq: Sn(varphi,T)(hat omega)}) we obtain
\begin{eqnarray}
(g_{n+1}-g_n)(\hat
\omega)=\sum_{i=\nu_n(\hat\omega)+1}^{\nu_{n+1}(\hat\omega)}\log
r_{i}(\hat\omega)-\psi(\hat
R^{n}\hat\omega)=
\sum_{i=\nu_n(\hat\omega)+1}^{\nu_{n+1}(\hat\omega)}\log\left(r_i(\hat\omega)\cdot\frac{(\hat
T^i\hat\omega)^-}{\xi_i(\hat\omega)}\right).\label{eq: gn+1-gn=sum...}
\end{eqnarray}
Our goal is to prove that the sequence $\{g_n\}_n$ converges exponentially fast. In order to show this, we will estimate $|g_{n+1}-g_n|$ by estimating each term of the sum in (\ref{eq: gn+1-gn=sum...}) for $n\geq4$.\\ 
Denoting $\tau=\tau( R^{n}(\hat\omega^+))$ and $j=\nu_n(\hat\omega)$, for $i=j+l$ and $1\leq l\leq \tau-1$ we have
\begin{eqnarray}
&&\log\left(r_i(\hat\omega)\cdot\frac{(\hat
T^i\hat\omega)^-}{\xi_i(\hat\omega)}\right)=
\log\left(\frac{2-\frac{1}{2-\frac{1}{2-\ldots-\frac{1}{2-\frac{1}{2-\beta}}}}\:\mbox{\tiny{(\virg{2} appears $l$ times)}}}{2-\frac{1}{2-\frac{1}{2-\ldots-\frac{1}{2-\frac{1}{2-\gamma}}}}\:\mbox{\tiny{(\virg{2} appears $l$ times)}}}\right)=\nonumber\\
&&=\log\left(\frac{(l+1)-l\,\beta}{(l+1)-l\,\gamma}\cdot\frac{l-(l-1)\gamma}{l-(l-1)\beta}\right)=\log\left(1+\zeta\right)-\log\left(1+\eta\right)\label{eq: estimate log(r T / e)}
\end{eqnarray}
where 
$$\beta=\les(k_j,\xi_{j-1}),(k_{j-1},\xi_{j-2}),\ldots,(k_2,\xi_1),(k_1,*)\res,$$ 
$$\gamma=\les(k_j,\xi_{j-1}),(k_{j-1},\xi_{j-2}),\ldots,(k_1,\xi_0),(k_0,\xi_{-1}),\ldots\res,$$ 
\begin{equation}\label{eq: xi and eta}
\zeta=\frac{l\,(\gamma-\beta)}{(l+1)-l\gamma}\hspace{.5cm}\mbox{and}\hspace{.5cm}\eta=\frac{(l-1)(\gamma-\beta)}{l-(l-1)\gamma}.
\end{equation} 
It is easy to see that $\frac{1}{3}\leq\gamma\leq\frac{3}{5}$ and from this we get  
\begin{equation}\label{eq: |xi-eta | leq}
\left|\zeta-\eta\right|=\frac{|\gamma-\beta|}{\left((l+1)-l\,\gamma\right)\left(l-(l-1)\gamma\right)}\leq\frac{25}{4\,l^2+16\,l+15}\left|\gamma-\beta\right|.
\end{equation}
Now we want to estimate $|\gamma-\beta|$ using (\ref{eq: growth of
hat q n}), i.e. in terms of $R$-convergents of $\gamma$. From its definition we know that 
$\beta$ is the  $j$-th $T$-convergent of
$\gamma$ but in general it is not a $R$-convergent. However
it can be shown that $j\geq\nu_{n-1}(\gamma)$ and therefore $q_{j}(\gamma)\geq
q_{\nu_{n-1}(\gamma)}(\gamma)=\hat q_{n-1}(\gamma)$. This, along with (\ref{eq: estimates for ECF}-\ref{eq: growth of hat q n}), leads us to $\left|\gamma-\beta\right|\leq\frac{1}{\hat q_{n-1}(\gamma)}\leq 3^{\frac{1-n}{3}}$ and we choose 
$n\geq4$ in order to guarantee that $\zeta,\eta>-1$. Now by (\ref{eq: estimate log(r T / e)}-\ref{eq: |xi-eta | leq})
we get
\begin{equation}\label{eq: |log(ri Ti / ei) | leq exp small for l}
\left|\log\left(r_i(\hat\omega)\cdot\frac{(\hat
T^i\hat\omega)^-}{\xi_i(\hat\omega)}\right)\right|\leq\frac{C_4}{4\,l^2+16\,l+15}\cdot3^{\frac{1-n}{3}},
\end{equation}
for some constant $C_4>0$. 
Now, let us estimate the last two terms in the sum in (\ref{eq: gn+1-gn=sum...}), i.e. $i=j+\tau$ and $i=j+\tau+1$.
We introduce
$\beta'=\les(k_i,\xi_{i-1}),(k_{i-1},\xi_{i-2}),\ldots,(k_2,\xi_1),(k_1,*)\res$
and $\gamma'=\les(k_i,\xi_{i-1}),(k_{i-1},\xi_{i-2}),\ldots,(k_1,\xi_0),(k_0,\xi_{-1}),\ldots\res$ and we observe that, since $\beta'$ is the $i$-th $T$-convergent of $\gamma'$, from (\ref{eq: estimates for ECF}) we get that
$|\gamma'-\beta'|\leq\frac{1}{q_i}\leq\frac{1}{i+1}\leq\frac{1}{2}$. This estimate and the fact that $2k_i-\gamma'\geq1$ give us
$$\log\left(r_{i}(\hat\omega)\cdot\frac{(\hat T^{i}\hat\omega)^-}{\xi_{i}(\hat\omega)}\right)=\log\left(\frac{2k_{i}-\beta'}{2k_{i}-\gamma'}\right)=\log\left(1+\frac{\gamma'-\beta'}{2k_i-\gamma'}\right)\hspace{.3cm}\mbox{and}$$
$$\left|\log\left(1+\frac{\gamma'-\beta'}{2k_i-\gamma'}\right)\right|\leq3\left|\gamma'-\beta'\right|.$$
Reasoning as above we get $|\gamma'-\beta'|\leq3^{\frac{1-n}{3}}$ when 
$n\geq4$ for $i=j+\tau$ and $i=j+\tau+1$. Therefore
\begin{equation}\label{eq: | r T/ e | leq exp small for last two i}
\sum_{i=j+\tau}^{j+\tau+1}\left|\log\left(r_{i}(\hat\omega)\cdot\frac{(\hat T^{i}\hat\omega)^-}{\xi_{i}(\hat\omega)}\right)\right|\leq2\cdot 3^{\frac{4-n}{3}}\leq 3^{\frac{7-n}{3}}.
\end{equation}
Now we estimate the sum in (\ref{eq: gn+1-gn=sum...}) using (\ref{eq: |log(ri Ti / ei) | leq exp small for l}-\ref{eq: | r T/ e | leq exp small for last two i}) and the convergence of $\sum_{l=1}^{\infty}(4\,l^2+16\,l+15)^{-1}$:
\begin{eqnarray}
\left|g_{n+1}-g_{n}\right|\leq C_4\,\sum_{l=1}^{\tau-1} \frac{3^{\frac{1-n}{3}}}{4\,l^2+16\,l+15}+3^{\frac{7-n}{3}}\leq C_5\,3^{-\frac{n}{3}},\hspace{.3cm}n\geq4,\nonumber
\end{eqnarray}
for some $C_5>0$ 
and uniformly in $\hat\omega$. The latter estimate allows us to define
\begin{equation}\label{eq: definition of g}
g(\hat\omega):=\sum_{n=0}^{\infty}(g_{n+1}-g_{n})(\hat\omega)
\end{equation}
and for $n\geq4$
\begin{equation}\nonumber
\sup_{{\hat\omega\in D(\hat R)}}\left|\varepsilon_n(\hat\omega)\right|=\sup_{\hat\omega\in D(\hat R)}\left|g(\hat\omega)-g_n(\hat\omega)\right|\leq C_5\sum_{m=n}^\infty3^{-\frac{m}{3}}\leq C_6\,3^{-\frac{n}{3}},
\end{equation}
for some $C_6>0$. 
The Lemma is therefore proven setting $C_3=C_6$.
\end{proof}
The following Lemma shows that on each bi-sided cylinder of length $n$ (for sufficiently large $n$) the function $g$ defined in (\ref{eq: definition of g}) can be approximated by a constant up to an error which is exponentially small with $n$. 
\begin{lem}\label{lem: exp estimate on a cylinder}
If $\hat\omega',\hat\omega''\in\mathcal{C}[c_{-n},\ldots,c_0;c_1,\ldots,c_n]$, $c_j\in\Sigma$, and $n\geq4$, then 
\begin{equation}\label{eq: lem absolute distance bound on cylinder}
\left|g(\hat\omega')-g(\hat\omega'')\right|\leq C_7\,3^{-\frac{n}{3}},
\end{equation}
where $C_7>0$ is an absolute constant.
\end{lem}
\begin{proof}
By Lemma \ref{lemma: approximation by Birkhoff sums} we have
\begin{equation}\label{eq; difference g(omega')-g(omega'')}
\left|g(\hat\omega')-g(\hat\omega'')\right|\leq C_6\,3^{-\frac{n}{3}}+\left|g_n(\hat\omega')-g_n(\hat\omega'')\right|,
\end{equation}
for every $n\geq4$. 
We need to estimate the second term in the right hand side of (\ref{eq; difference g(omega')-g(omega'')}). Let $\hat\omega'=\left\{(k_i',\xi_i')\right\}_{i\in\Z}$ and $\hat\omega''=\left\{(k_i'',\xi_i'')\right\}_{i\in\Z}$ with $(k_i',\xi_i'), (k_i'',\xi_i'')\in\Omega$. By assumption we have $\nu_j(\hat\omega')=\nu_j(\hat\omega'')=\nu_j$ for $-n-1\leq j\leq n$ and $(k_i',\xi_i')=(k_i'',\xi_i'')=(k_i,\xi_i)$ for $\nu_{-n-1}\leq i\leq\nu_{n+1}-1$. From (\ref{eq: Sn(varphi,T)(hat omega)}) we get
\begin{eqnarray}
S_n(\psi,\hat R)(\hat\omega')-S_n(\psi,\hat R)(\hat\omega'')=\sum_{j=0}^{n-1}\sum_{i=\nu_j+1}^{\nu_{j+1}}\log\left(\frac{(\hat T^i\hat\omega'')^-}{(\hat T^i\hat\omega')^-}\right)\label{eq: Sn(omega')-Sn(omega'')}
\end{eqnarray}
The estimate of the sum indexed by $i$ is now done with the same technique used in the proof of Lemma \ref{lemma: approximation by Birkhoff sums}.
Denoting $\tau=\tau(R^j((\hat\omega')^+))$, $\nu=\nu_j(\hat\omega')$ and $\bar\nu=-\nu_{-n-1}(\hat\omega)$, for $i=\nu+l$ and $1\leq l\leq\tau-1$ we have
\begin{eqnarray}
\log\left(\frac{(\hat T^i\hat\omega'')^-}{(\hat T^i\hat\omega')^-}\right)=\log(1+\zeta)-\log(1+\eta),\nonumber
\end{eqnarray}
where $\zeta$ and $\eta$ are as in (\ref{eq: xi and eta}), with 
\begin{eqnarray}
\beta&=&\les(k_\nu,\xi_{\nu-1}),
\ldots,(k_{-\bar\nu+1},\xi_{-\bar\nu}),(k_{-\bar\nu},\xi'_{-\bar\nu-1}),(k'_{-\bar\nu-1},\xi'_{-\bar\nu-2}),\ldots\res\hspace{.2cm}\mbox{and}\nonumber\\
\gamma&=&\les(k_\nu,\xi_{\nu-1}),
\ldots,(k_{-\bar\nu+1},\xi_{-\bar\nu}),(k_{-\bar\nu},\xi''_{-\bar\nu-1}),(k''_{-\bar\nu-1},\xi''_{-\bar\nu-2}),\ldots\res.\nonumber
\end{eqnarray}
Since (\ref{eq: |xi-eta | leq}) still holds, we want to estimate $|\gamma-\beta|$. Our assumptions imply that $\beta$ and $\gamma$ share the same ECF-expansion up to the index $\nu+\bar\nu+1$ and it can be shown that $\nu+\bar\nu+1\geq\nu_{n+j-1}(\beta)$. Therefore by (\ref{eq: growth of
hat q n}) we get $|\gamma-\beta|\leq\frac{1}{q_{\nu+\bar\nu+1}(\beta)}\leq\frac{1}{\hat q_{n+j-1}(\beta)}\leq 3^{\frac{1-n-j}{3}}$. We choose $n\geq4$ as before and by (\ref{eq: Sn(omega')-Sn(omega'')}) and (\ref{eq: |xi-eta | leq}) we find, for $i=\nu+l$,
\begin{equation}\nonumber
\left|\log\left(\frac{(\hat T^i\hat\omega'')^-}{(\hat T^i\hat\omega')^-}\right)\right|\leq\frac{C_8}{4\,l^2+16\,l+15}\,3^{\frac{1-n-j}{3}},
\end{equation}
for some constant $C_8>0$. The last two terms, corresponding to $i=\nu+\tau$ and $i=\nu+\tau+1$, are estimated in the same way, obtaining 
$\left|\log\left(\frac{(\hat T^i\hat\omega'')^-}{(\hat T^i\hat\omega')^-}\right)\right|\leq 3^{\frac{4-n-j}{3}}$. Therefore
\begin{equation}\nonumber
\sum_{i=\nu_j+1}^{\nu_{j+1}}\left|\log\left(\frac{(\hat T^i\hat\omega'')^-}{(\hat T^i\hat\omega')^-}\right)\right|\leq C_8\,\sum_{l=1}^{\tau-1}\frac{3^{\frac{1-n-j}{3}}}{4\,l^2+16\,l+15}+3^{\frac{7-n-j}{3}}\leq C_9\,3^{\frac{1-n-j}{3}},
\end{equation}
for some $C_9>0$ and hence
\begin{equation}\nonumber
\left|S_n(\psi,\hat R)(\hat\omega')-S_n(\psi,\hat R)(\hat\omega'')\right|\leq C_9\,\sum_{j=0}^{n-1}3^{\frac{1-n-j}{3}}\leq C_{10}\, 3^{-\frac{n}{3}},
\end{equation}
for some $C_{10}>0$.
Now by (\ref{def: gn}), since $\hat q_n(\hat\omega')=\hat q_n(\hat\omega'')$,  we get $\left|g_n(\hat\omega')-g_n(\hat\omega'')\right|\leq C_{10} \,3^{-\frac{n}{3}}$ and therefore, setting $C_7=C_6+C_{10}$, we get (\ref{eq: lem absolute distance bound on cylinder}), as claimed.

\end{proof}
\subsection{Comparing Renewal Times}
Given $\hat\omega$ and $L$, we want to choose $T$ as a function of $L$ in order to compare $\hat n_L(\hat\omega)$ and $r(\hat\omega,T)$. Recall that $\hat n_L(\hat\omega)$ is uniquely determined by 
\begin{equation}\label{eq: n_L uniquely determined}
\log\hat q_{\hat n_L(\hat\omega)-1}\leq\log L<\log\hat q_{\hat n_L(\hat\omega)}.
\end{equation}
Using (\ref{def: gn}), we rewrite (\ref{eq: n_L uniquely determined}) as
\begin{equation}\label{eq: S_n_L-1 + g_n_L-1 leq log L < S_n_L + g_n_L}
S_{\hat n_L(\hat\omega)-1}(\psi)(\hat\omega)+g_{\hat n_L(\hat\omega)-1}(\hat\omega)\leq\log L<S_{\hat n_L(\hat\omega)}(\psi)(\hat\omega)+g_{\hat n_L(\hat\omega)}(\hat\omega)
\end{equation}
after dropping the dependence on $\hat R$ in the notation for the Birkhoff sums. To avoid the dependence of time $T$ on $\hat\omega$, let us consider a set $\mathcal{C}\subset D(\hat R)$ and denote $g_{\mathcal{C}}:=\sup_{\hat\omega\in\mathcal{C}}g(\hat\omega)$. Assume that all $\hat\omega\in\mathcal{C}$ satisfy $\left|g(\hat\omega)-g_{\mathcal{C}}\right|\leq\frac{\varepsilon}{2}$. We shall deal with such sets in the proof Theorem \ref{theorem: renewal for hat q}. 

The following Lemma guarantees that $\hat n_L(\hat\omega)$ grows uniformily when $L$ grows for $\hat\omega$ belonging to a set of sufficiently large measure. This fact is not obvious because a priori the $R$-denominators $\hat q_n(\hat\omega)$ might grow very fast for some $\hat\omega$ and for that reason $L\mapsto\hat n_L(\hat\omega)$ might be very slowly increasing. However, using Lemma \ref{lem: growth of hat q n, upper bound}, we prove that this cannot happen on a set of large measure.
\begin{lem}\label{lem: n_L grows with L}
For each measurable $\mathcal{C}\subset D(\hat R)$ and $\varepsilon>0$, there exists a measurable set $\mathcal{C}'\subseteq\mathcal{C}$ such that $\hat\mu(\mathcal{C}\smallsetminus\mathcal{C}')\leq\varepsilon\,\hat\mu(\mathcal{C})$ and $\min_{\hat\omega\in\mathcal{C}'}\hat n_L(\hat\omega)\rightarrow\infty$ uniformly as $L\rightarrow\infty$.\\
Likewise, given $\varepsilon>0$, there exists a measurable set $\mathcal{C}_\varepsilon\subseteq(0,1]$ such that $\mu\left((0,1]\smallsetminus\mathcal{C}_\varepsilon\right)\leq\varepsilon$ and $\min_{\hat\omega\in\mathcal{C}_\varepsilon}\hat n_L(\hat\omega)\rightarrow\infty$ uniformly as $L\rightarrow\infty$.
\end{lem}
\begin{proof}
Let us consider $\overline{\mathcal{C}}=\pi(\mathcal{C})$ where $\pi:D(\hat R)\rightarrow\X$ is the natural projection.
By Lemma \ref{lem: growth of hat q n, upper bound} we know that there exists a set $\overline{\mathcal{C}}_1\subseteq\overline{\mathcal{C}}$, with $\mu(\overline{\mathcal{C}}\smallsetminus\overline{\mathcal{C}}_1)=0$, and $\overline{n}\in\N$ such that for every $n>\overline{n}$ and every $\hat\omega^+\in\overline{\mathcal{C}}_1$ we have $\frac{\log\hat q_n(\hat\omega^+)}{n}\leq C_{1}$. Moreover there exists a set $\overline{\mathcal{C}}_2\subseteq\overline{\mathcal{C}}$, with $\mu(\overline{\mathcal{C}}\smallsetminus\overline{\mathcal{C}}_2)\leq\varepsilon\mu(\overline{\mathcal{C}})$ and a constant $C_{11}=C_{11}(\overline{\mathcal{C}},\varepsilon)$ such that for every $\hat\omega^+\in\overline{\mathcal{C}}_2$ we have $\frac{\log\hat q_j(\hat\omega^+)}{j}\leq C_{11}$ for $j=1,\ldots,\overline{n}$. Setting $\overline{\mathcal{C}}'=\overline{\mathcal{C}}_1\cap\overline{\mathcal{C}}_2$ and $C_{12}=C_{12}(\overline{\mathcal{C}},\varepsilon)=\max\{C_{1},C_{11}\}$ we have $\mu(\overline{\mathcal{C}}\smallsetminus\overline{\mathcal{C}}')\leq\varepsilon\mu(\overline{\mathcal{C}})$
and $\hat q_n\leq C_{12}\,n$ for every $\hat\omega^+\in\overline{\mathcal{C}}'$. 
Since by construction all the functions $\hat q_n(\cdot)$ and $\hat n_L(\cdot)$ are constant on the fibers $\pi^{-1}\alpha$, $\alpha\in\X$, setting $\mathcal{C}'=\pi^{-1}\overline{\mathcal{C}}'$, the same statement is true for all $\hat\omega\in\mathcal{C}'$ and $\hat\mu(\mathcal{C}\smallsetminus\mathcal{C}')\leq\varepsilon\,\hat\mu(\mathcal{C})$. 
By definition $\hat q_{\hat n_L(\hat\omega)}>L$ and this implies that $\min_{\hat\omega\in\mathcal{C}'}\big(\hat n_L(\hat\omega)%\log \hat n_L(\hat\omega)
\big)\geq\frac{1}{C_{12}}\log\hat q_{\hat n_L(\hat\omega)}\geq\frac{\log L}{C_{12}}$. This proves the
% From the monotonicity of $t\mapsto t\log t$ we get the 
first part of the Lemma. The second part is proven in the same way. 
\end{proof}
The following Lemma considers a bi-sided cylinder $\mathcal{C}$ and shows that for a suitable choice of $T$ as a function of $L$ and $\mathcal{C}$, the two quantities $\hat n_L(\hat\omega)$ and $r(\hat\omega,T)$ coincide on a subset of $\mathcal{C}$ with relatively large measure.
\begin{lem}\label{lem: n_L = r(omega,T) on a set of large measure}
Let $n\geq4$ and $\varepsilon>0$. Let us consider $\mathcal{C}\in\frak{C}_{n,n}$ and assume that $\left|g(\hat\omega)-g_\mathcal{C}\right|\leq\frac{\varepsilon}{2}$ for all $\hat\omega\in\mathcal{C}$. Define $T=T(L,\mathcal{C}):=\log L-g_\mathcal{C}$ and
$U=U(\mathcal{C}):=\left\{\hat\omega\in\mathcal{C}:\:\hat n_{L}(\hat\omega)\neq r(\hat\omega,T)\right\}.$ Then there exists $L_0=L_0(\mathcal{C})>0$ such that, for all $L\geq L_0$, we have $\hat\mu(U)\leq7\varepsilon\,\hat\mu(\mathcal{C})$. 
\end{lem}
We shall provide only a sketch of the proof of Lemma \ref{lem: n_L = r(omega,T) on a set of large measure}, since it is similar to the proof of Lemma 3.4 in \cite{Sinai-Ulcigrai07}.
\begin{proof}[Sketch of the proof]
Let $T=T(L,\mathcal{C})=\log L-g_{\mathcal{C}}$. By definition we have
\begin{equation}\label{S_r-1 leq T< S_r}
S_{r(\hat\omega,T)-1}(\psi)(\hat\omega)\leq T=\log L-g_{\mathcal{C}}<S_{r(\hat\omega,T)}(\psi)(\hat\omega).
\end{equation}
Let $\mathcal{C}'\subseteq\mathcal{C}$ be as in Lemma \ref{lem: n_L grows with L} and define two sets $U_{\pm\varepsilon}\subset D(\hat R)$
as
\begin{eqnarray}
\hspace{.3cm}U_{-\varepsilon}:=\left\{\hat\omega:\:T<S_{r(\hat\omega,T)}(\psi)(\hat\omega)\leq T+\varepsilon\right\}\hspace{.3cm}\mbox{and}\hspace{.3cm}
U_{\varepsilon}:=\left\{\hat\omega:\:T-\varepsilon<S_{r(\hat\omega,T)-1}(\psi)(\hat\omega)\leq T\right\}.\nonumber
\end{eqnarray}
Using (\ref{eq: S_n_L-1 + g_n_L-1 leq log L < S_n_L + g_n_L}-\ref{S_r-1 leq T< S_r}) and Lemmata \ref{lemma: approximation by Birkhoff sums} and 
\ref{lem: n_L grows with L} (we are assuming $n\geq4$) it is possible to show that
for some $L_0>0$ and every $L\geq L_0$ we have
\begin{equation}\nonumber%\label{eq: inclusion U cap C' subset (U eps cup U - eps) cap C'}
U\cap\mathcal{C}'\subseteq (U_{\varepsilon}\cup U_{-\varepsilon})\cap\mathcal{C}'.
\end{equation}
According to the definition (\ref{eq: action special flow Phi_t}) of the special flow, the sets $U_{\pm\varepsilon}$ can be rewritten as 
$$U_{\varepsilon}=\left\{(\hat\omega,0):\Phi_T(\hat\omega,0)\in D_\Phi^{\varepsilon}\right\}
\hspace{.5cm}\mbox{and}\hspace{.5cm}
U_{-\varepsilon}=\left\{(\hat\omega,0):\Phi_T(\hat\omega,0)\in D_\Phi^{-\varepsilon}\right\},$$
where $D_\Phi^{\varepsilon}:=D(\hat R)\times[0,\varepsilon)$ and $D_\Phi^{-\varepsilon}:=\left\{(\hat\omega,y):\:\psi(\hat\omega)-\varepsilon\leq y<\psi(\hat\omega)\right\}$. Our aim now is to use mixing of $\{\Phi_t\}_{t\in\mathbb{R}}$ in order to estimate the measures of $U_{\pm\varepsilon}$. 
Since our special flow is 3-dimensional, we need to ``thicken'' them as follows. Using Lemma \ref{lemma: delta and M} (since $n\geq4$) we can choose $0<\delta\leq\varepsilon$ such that $\delta<\min_{\hat\omega\in\mathcal{C}}\psi(\hat\omega)$ and construct two subsets of $D_\Phi$:
\begin{equation}\nonumber
U_{\pm\varepsilon}^\delta:=\left\{(\hat\omega,z):\:0\leq z<\delta,\:\Phi_T(\hat\omega,z)\in D_\Phi^{\pm\varepsilon}\right\}=D_\Phi^\delta\cap\Phi_{-T}\left(D_\Phi^{\pm\varepsilon}\right).
\end{equation}
Again using the definition (\ref{eq: action special flow Phi_t}) it is possible to show that $\left(U_{\varepsilon}\cap\mathcal{C}\right)\times[0,\delta)\subseteq U_{\varepsilon+\delta}^\delta$ and $\left(U_{-\varepsilon}\cap\mathcal{C}\right)\times[0,\delta)\subseteq U_{-\varepsilon}^\delta\cup U_\delta^\delta$.
Define now $\mathcal{C}'_\delta:=\mathcal{C}'\times[0,\delta)$. So far we proved
\begin{eqnarray}
(U\cap\mathcal{C}')\times[0,\delta)&\subseteq&\big((U_{+\varepsilon}\cup U_{-\varepsilon})\cap\mathcal{C}'\big)\times[0,\delta)\subseteq\mathcal{C}_\delta'\cap\big(U_{-\varepsilon}^\delta\cup U_{\delta+\varepsilon}^\delta\big)=\nonumber\\
&=&\mathcal{C}_\delta'\cap\Phi_{-T}\big(D_\Phi^{-\varepsilon}\cup D_\Phi^{\delta+\varepsilon}\big).\nonumber
\end{eqnarray}
From the previous inclusions and the mixing property of the special flow (Proposition \ref{prop: the flow is mixing}), one can find some $T_0>0$ such that, for any $T\geq T_0$, we have
\begin{eqnarray}
\hat\mu(U\cap\mathcal{C}')\cdot\delta&=&\tilde\mu\big((U\cap\mathcal{C}')\times[0,\delta)\big)\leq\tilde\mu\big(\mathcal{C}'_\delta\cap\Phi_{-T}(D_\Phi^{-\varepsilon}\cup D_\Phi^{\delta+\varepsilon})\big)\leq\nonumber\\
&\leq&2\,\hat\mu(\mathcal{C}')\cdot\delta\cdot\tilde\mu(D_\Phi^{-\varepsilon}\cup D_\Phi^{\delta+\varepsilon})\leq2\,\hat\mu(\mathcal{C}')\cdot\delta\cdot3\,\varepsilon,\label{eq: estimate of hat mu measure of U cap C'}
\end{eqnarray}
where the last inequality follows from the estimate $\tilde\mu(D_\Phi^{\pm\varepsilon})\leq\varepsilon$. Therefore we get $\hat\mu(U\cap\mathcal{C}')\leq 6\varepsilon\,\hat\mu(\mathcal{C}')$. We can enlarge $L_0$ so that $\log L_0-g_\mathcal{C}\geq T_0$, $L\geq L_0$ implies $T\geq T_0$ and (\ref{eq: estimate of hat mu measure of U cap C'}) still holds. Thus, $\hat\mu(U)\leq\hat\mu(U\smallsetminus\mathcal{C}')+6\varepsilon\,\hat\mu(\mathcal{C})\leq 7\varepsilon\,\hat\mu(\mathcal{C})$ and the Lemma is proven.
\end{proof}

\section{Proof of the Existence of the Limiting Distribution}\label{section: existence of limiting distribution}
Recall that for $L>0$ we defined $\hat n_L(\hat\omega)=\min\left\{n\in\N:\:\hat q_n(\hat\omega)>L\right\}$.
\begin{proof}[Proof of Theorem \ref{theorem: renewal for hat q}%
]
Assume $1<a<b$ and $c_j\in\Sigma$, $-N_1<j\leq N_2$. Our aim is to estimate the expression in (\ref{eq: main theorem}). Since the quantities $q_i(\hat\omega)$ (and in particular $\hat q_i(\hat\omega)$) and $\hat n_L(\hat\omega)$ depend only on $\alpha=\hat\omega^+$, for $\hat n_L(\omega)>N_1$, we can rewrite the condition $\sigma_{\hat n_L+j}=c_j$, $-N_1< j\leq N_2$, as
\begin{equation}\label{eq: equivalent formulation}
\hat R^{\hat n_L(\hat\omega)-1}(\hat\omega)\in\mathcal{C}_{N_1,N_2},\hspace{.3cm}\mbox{where}\hspace{.3cm}\mathcal{C}_{N_1,N_2}:=\hat R^{N_1-1}\left(\hat{\mathcal{C}}[c_{-N_1+1},c_{-N_1+2},\ldots,c_0,\ldots,c_{N_2}]\right).
\end{equation}
Given two functions $F_1, F_2$ on $D(\hat R)$, we define $D_\Phi(F_1,F_2):=\big\{(\hat\omega,y)\in D_\Phi:\:\psi(\hat\omega)-F_2(\hat\omega)<y<\psi(\hat\omega)-F_1(\hat\omega)\big\}$.
Notice that for some values of $F_1(\hat\omega)$ and $F_2(\hat\omega)$, the corresponding set of $y$ can be empty. Moreover, let us remark that if $F'_1\leq F_1$ and $F'_2\geq F_2$, then $D_\Phi(F_1,F_2)\subseteq D_\Phi(F'_1,F'_2)$. Define $p:D_\Phi\rightarrow D(\hat R)$, $p(x,y)=x$ the projection on the base of the special flow.
We shall show that the limiting distribution $\mathrm{P}_{N_1,N_2}'$ exists and it is given by
\begin{equation}\label{eq: P_N= tilde mu(...)}
\mathrm{P}_{N_1,N_2}'\big((a,b)\times\{c_{-N_1+1}\}\times\cdots\{c_0\}\times\cdots\times\{c_{N_2}\}\big)=\tilde\mu\big(D_\Phi(\log a,\log b)\cap p^{-1}\mathcal{C}_{N_1,N_2}\big).
\end{equation}
Consider $\varepsilon>0$. For each $n\in\N$, the collection $\{\mathcal{C}:\:\mathcal{C}\in\frak{C}_{n,n}\}$ is a countable partition of $D(\hat R)$. Let us choose $n\geq 4$ so that, by Lemma \ref{lem: exp estimate on a cylinder}, we have $\left|g(\hat\omega')-g(\hat\omega'')\right|\leq\frac{\varepsilon}{2}$ for all $\hat\omega',\hat\omega''\in\mathcal{C}$. Define also $A_\mathcal{C}:=\big\{\hat\omega\in\mathcal{C}:\: a<\frac{\hat q_{\hat n_L(\hat\omega)}(\hat\omega)}{L}<b,\:\hat R^{\hat n_L(\hat\omega)-1}(\hat\omega)\in\mathcal{C}_{N_1,N_2}\big\}$. By the second part of Lemma \ref{lem: n_L grows with L}, there exists $L_1>0$ such that $\hat n_L(\alpha)>N$ for every $L\geq L_1$ and every $\alpha$ in the complement of a set of $\mu$-measure less than $\varepsilon$. Hence, (\ref{eq: equivalent formulation}) gives us
\begin{equation}\nonumber
\left|\mu\left(\left\{\alpha:\: a<\frac{\hat q_{\hat n_L(\alpha)}(\alpha)}{L}<b,\:\sigma_{\hat n_L(\alpha)+j}=c_j,\:-N_1< j\leq N_2\right\}\right)-\sum_{\mathcal{C}\in\frak{C}_{n,n}}\hat\mu(A_\mathcal{C})\right|\leq2\varepsilon.
\end{equation}
Let us consider the finite collection of cylinders $\frak{C}_{n,n}^{h,m}$ whose elements $\mathcal{C}=\mathcal{C}[c_{-n},\ldots;\ldots,c_n]$ 
are such that $c_i=h_i\cdot m_i^\pm\in\Sigma$ with $h_i<h$ and $m_i<m$ for $-n\leq i\leq n$. It is clear that if $\mathcal{C}\in\frak{C}_{n,n}\smallsetminus\frak{C}_{n,n}^{h,m}$, then there exist $-n\leq i\leq n$ %and $j\geq m$ 
such that $h_i>h$ or $m_i>m$ and therefore, by the $\hat R$-invariance of the measure $\hat\mu$ and Remark \ref{remark: measure of cylinders}, we get
\begin{eqnarray}%\label{eq: estimate sum of measures of unbounded cylinders}
&&\sum_{\mathcal{C}\in\frak{C}_{n,n}\smallsetminus\frak{C}_{n,n}^{h,m}}\hat\mu(\mathcal{C})\leq\sum_{i=-n}^n\sum_{\tiny{
\left.\begin{array}{c}l\cdot j^{\pm}\in\Sigma\\l\geq h\:\,\mbox{or}\:\,j\geq m\end{array}\right.}}
\hat\mu\left(\hat R^i\big(\hat{\mathcal{C}}[l\cdot j^\pm]\big)\right)\leq%
% \leq\frac{6(2n+1)}{\log3}\left(\sum_{l\geq h}
%\sum_{j\geq 1}+\sum_{l\geq0}
%\sum_{j\geq m}\right)\frac{1}{4l^2+8l+3}\frac{1}{j^2}
\nonumber\\
&&\leq\frac{6(2n+1)}{\log3}\left(\sum_{l\geq h}
\sum_{j\geq 1}+\sum_{l\geq0}
\sum_{j\geq m}\right)\frac{1}{4l^2+8l+3}\frac{1}{j^2}\leq\nonumber\\%\sum_{l\geq h}\mathcal{O}\left(\frac{1}{l^2}\right)+\sum_{j\geq m}\mathcal{O}\left(\frac{1}{j^2}\right)=\nonumber\\%\mathcal{O}\left(\frac{1}{m}\right).\\
&&\leq C_{13}\sum_{l\geq h}\frac{1}{4l^2+8l+3}+C_{14}\sum_{j\geq m}\frac{1}{j^2}\leq C_{15}\left(\frac{1}{h}+\frac{1}{m}\right),%\mathcal{O}\left(\frac{1}{h}\right)+\mathcal{O}\left(\frac{1}{m}\right).
\label{eq: measure of cylinders in C_n-C_n^m}
\end{eqnarray}
for some $C_{13},C_{14},C_{15}>0$. Thus, it is possible to choose $h$ and $m$ sufficiently large, so that $\sum_{\mathcal{C}\in\frak{C}_{n,n}\smallsetminus\frak{C}_{n,n}^{h,m}}\hat\mu(\mathcal{C})\leq\varepsilon$.
For each $\mathcal{C}\in\frak{C}_{n,n}^{h,m}$ we can find $L_0(\mathcal{C})$ and $U(\mathcal{C})$ as in  Lemma \ref{lem: n_L = r(omega,T) on a set of large measure} and for every $L\geq\max_{\mathcal{C}\in\frak{C}_{n,n}\smallsetminus\frak{C}_{n,n}^{h,m}}L_0(\mathcal{C})$, using also the inclusion $A_\mathcal{C}\subseteq\mathcal{C}$ and (\ref{eq: measure of cylinders in C_n-C_n^m}), we get 
\begin{eqnarray}
\left|\mu\left(\left\{\alpha:\:a<\frac{\hat q_{\hat n_L(\alpha)}(\alpha)}{L}<b,\:\sigma_{\hat n_L(\alpha)+j}=c_j,\:-N_1< j\leq N_2\right\}\right)-\sum_{\mathcal{C}\in\frak{C}_{n,n}^{h,m}}\hat\mu\left(A_{\mathcal{C}\smallsetminus U(\mathcal{C})}\right)\right|\leq\nonumber\\
\leq2\varepsilon+\left|\sum_{\mathcal{C}\in\frak{C}_{n,n}\smallsetminus\frak{C}_{n,n}^{h,m}}\hat\mu\left(A_\mathcal{C}\right)+\sum_{\mathcal{C}\in\frak{C}_{n,n}^{h,m}}\hat\mu\left(A_{\mathcal{C}\cap U(\mathcal{C})}\right)\right|\leq 3\varepsilon+7\varepsilon\sum_{\mathcal{C}\in\frak{C}_{n,n}^{h,m}}\hat\mu(\mathcal{C})\leq10\varepsilon.\nonumber
\end{eqnarray}
In order to get (\ref{eq: P_N= tilde mu(...)}), it is enough to prove that, for each $\mathcal{C}\in\frak{C}_{n,n}^{h,m}$ and sufficiently large $L$, we have
\begin{equation}\label{eq: estimate for finite cylinders in C_n^m}
\left|\frac{\hat\mu\left(A_{\mathcal{C}\smallsetminus U(\mathcal{C})}\right)}{\hat\mu\left(\mathcal{C}\smallsetminus U(\mathcal{C})\right)}-\tilde\mu\big(D_\Phi(\log a,\log b)\cap p^{-1}\mathcal{C}_{N_1,N_2}\big)\right|\leq C_{16}\,\varepsilon,
\end{equation}
for some $C_{16}>0$. Let $\mathcal{C}\in\frak{C}_{n,n}^{h,m}$ be fixed, consider $T=\log L-g_\mathcal{C}$, and let $U=U(\mathcal{C})$ be as in 
Lemma \ref{lem: n_L = r(omega,T) on a set of large measure}. By the same Lemma, for $L\geq L_0(\mathcal{C})$, we have $\hat n_L(\hat\omega)=r(\hat\omega,T)$ for every $\hat\omega\in\mathcal{C}\smallsetminus U$. Using Lemma \ref{lemma: approximation by Birkhoff sums} we get
\begin{eqnarray}
&&\hspace{-.6cm}\left\{\hat\omega\in\mathcal{C}\smallsetminus U:\:a<\frac{\hat q_{\hat n_L(\hat\omega)}(\hat\omega)}{L}<b\right\}=\left\{\hat\omega\in\mathcal{C}\smallsetminus U:\:\log a<\log\hat q_{r(\hat\omega,T)}(\hat\omega)-\log L<\log b\right\}=\nonumber\\
&&\hspace{-.3cm}=\left\{\hat\omega\in\mathcal{C}\smallsetminus U:\:\log a<S_{r(\hat\omega,T)}(\psi)(\hat\omega)-T+\varepsilon_{L,\mathcal{C}}(\hat\omega)<\log b\right\},\nonumber
\end{eqnarray}
where $\varepsilon_{L,\mathcal{C}}(\hat\omega):=\varepsilon_{\hat n_L(\hat\omega)}(\hat\omega)-g_\mathcal{C}+g(\hat\omega)$ and $\varepsilon_{\hat n_L(\hat\omega)}(\hat\omega)$ is defined as in (\ref{def: gn}-\ref{eq: statement lemma birkhoff sum denominators}). It is possible to show, using Lemma \ref{lemma: delta and M} and (\ref{S_r-1 leq T< S_r}), that $\left|\varepsilon_{L,\mathcal{C}}(\hat\omega)\right|\leq 2\varepsilon$ uniformly on $\mathcal{C}\smallsetminus U$ (see \cite{Sinai-Ulcigrai07} and the proof of Theorem 1.1 therein). Denoting by $v(\Phi_t(x,y))$ the vertical component $y'$ of $\Phi_t(x,y)=(x',y')$, by (\ref{eq: action special flow Phi_t}) and the equality $\hat n_{L}(\hat\omega)=r(\hat\omega,T)$, we get
$S_{r(\hat\omega,T)}(\psi)(\hat\omega)-T=\psi\left(\hat R^{\hat n_L(\hat\omega)-1}(\hat\omega)\right)-v(\Phi_T(\hat\omega,0))$,
which represent the vertical distance from $\Phi_T(\hat\omega,0)$ and the roof function. Observing that $\hat R^{\hat n_L(\hat\omega)-1}(\hat\omega)=p(\Phi_T(\hat\omega,0))$, the condition (\ref{eq: equivalent formulation}) can be rewritten as $p(\Phi_T(\hat\omega,0))\in\mathcal{C}_{N_1,N_2}$ and, recalling the inclusion properties of the sets $D_\Phi(F_1,F_2)$, we get 
\begin{eqnarray}
A_{\mathcal{C}\smallsetminus U}&\subseteq&\big(\mathcal{C}\smallsetminus U\times\{0\}\big)\cap\Phi_{-T}\big(D_{\Phi}(\log a-2\varepsilon,\log b+2\varepsilon)\cap p^{-1}\mathcal{C}_{N_1,N_2}\big)\hspace{.3cm}\mbox{and}\nonumber\\
A_{\mathcal{C}\smallsetminus U}&\supseteq&\big(\mathcal{C}\smallsetminus U\times\{0\}\big)\cap\Phi_{-T}\big(D_{\Phi}(\log a+2\varepsilon,\log b-2\varepsilon)\cap p^{-1}\mathcal{C}_{N_1,N_2}\big).\nonumber
\end{eqnarray}
Now we use the same strategy used in the proof of Lemma \ref{lem: n_L = r(omega,T) on a set of large measure}, ``thickening'' the sets and applying the mixing of the special flow $\{\Phi_t\}_{t\in\mathbb{R}}$ (see again \cite{Sinai-Ulcigrai07} for details). By Lemma \ref{lemma: delta and M} we can choose $0<\delta<\min\{\min_{\hat\omega\in\mathcal{C}}\psi(\hat\omega),\varepsilon\}$ and for each $\hat\omega\in A_{\mathcal{C}\smallsetminus U}$ and $0\leq z<\delta$ it is possible to show that $(\hat\omega,z)\in\Phi_{-T}\big(D_\Phi(\log a-2\varepsilon-\delta,\log b+2\varepsilon)\cap p^{-1}\mathcal{C}_{N_1,N_2}\cup D_\Phi^\delta\big)$, where $D_\Phi^\delta=D(\hat R)\times[0,\delta)$. 
We get
\begin{eqnarray}
\delta\cdot\hat\mu\left(A_{\mathcal{C}\smallsetminus U}\right)\leq\tilde\mu\Big(\big(\mathcal{C}\smallsetminus U\times[0,\delta)\big)\cap\Phi_{-T}\big(D_\Phi(\log a-3\varepsilon,\log b+2\varepsilon)\cap p^{-1}\mathcal{C}_{N_1,N_2}\cup D_\Phi^\delta\big)\Big)\leq\nonumber\\
\leq\delta\cdot\hat\mu(\mathcal{C}\smallsetminus U)\cdot\Big(\tilde\mu\big(D_\Phi(\log a-3\varepsilon,\log b+2\varepsilon)\cap p^{-1}\mathcal{C}_{N_1,N_2}\big)+2\varepsilon\Big),\label{eq: estimate using mixing 1}
\end{eqnarray} 
where (\ref{eq: estimate using mixing 1}) follows from the mixing property of the flow (Proposition \ref{prop: the flow is mixing}) after observing that $\tilde\mu(D_\Phi^\delta)\leq\varepsilon$ and possibly enlarging $L_0$ so that if $L\geq L_0$, also $T=T(L,\mathcal{C})$ is sufficiently large. 
Again, reasoning as in the proof of Lemma \ref{lem: n_L = r(omega,T) on a set of large measure}, it is possible to show that for each $\hat\omega\in\mathcal{C}\smallsetminus U$ and $0\leq z<\delta$  such that $(\hat\omega,z)\in\Phi_{-T}\big(D_\Phi(\log a+2\varepsilon,\log b-2\varepsilon-\delta)\cap p^{-1}\mathcal{C}_{N_1,N_2}\smallsetminus D_\Phi^\delta\big)$, we have $\hat\omega\in A_{\mathcal{C}\smallsetminus U}$. This implies that
\begin{equation}\nonumber
\mathcal{C}\smallsetminus U\times[0,\delta)\cap\Phi_{-T}\big(D_\Phi(\log a+2\varepsilon,\log b-3\varepsilon)\cap p^{-1}\mathcal{C}_{N_1,N_2}\smallsetminus D_\Phi^\delta\big)\subseteq A_{\mathcal{C}\smallsetminus U}\times[0,\delta).
\end{equation}
Remark that for any measurable $D\subseteq D_\Phi$, we have $\tilde\mu(D\smallsetminus D_\Phi^\delta)\geq\tilde\mu(D)-\varepsilon$. Now, using mixing and enlarging $L_0$  if needed, for $L\geq L_0$ we get
\begin{equation}\label{eq: estimate using mixing 2}
\delta\cdot\hat\mu\left(A_{\mathcal{C}\smallsetminus U}\right)\geq\delta\cdot\hat\mu\left(\mathcal{C}\smallsetminus U\right)\cdot\Big(\tilde\mu\big(D_\Phi(\log a+2\varepsilon,\log b-3\varepsilon)\cap p^{-1}\mathcal{C}_{N_1,N_2}\big)-2\varepsilon\Big).
\end{equation}
Moreover, by Fubini Theorem, for $(C_{17},C_{18})\in\{(-3,2),(2,-3)\}$,
\begin{equation}\label{eq: estimate Fubini}
\big|\tilde\mu\big(D_\Phi(\log a+ C_{17}\,\varepsilon,\log b+ C_{18}\,\varepsilon)\cap p^{-1}\mathcal{C}_{N_1,N_2}\big)-\tilde\mu\big(D_\Phi(\log a,\log b)\cap p^{-1}\mathcal{C}_{N_1,N_2}\big)\big|\leq C_{19}\,\varepsilon,
\end{equation}
for some $C_{19}>0$.
Finally, by (\ref{eq: estimate using mixing 1}-\ref{eq: estimate Fubini}) we get (\ref{eq: estimate for finite cylinders in C_n^m}) concluding thus the proof of the existence of the limiting distribution.
\end{proof}
\begin{remark}
The set $\mathcal{C}_{N_1,N_2}$ in the previous proof can be replaced by any set of positive $\hat\mu$ measure in the base $D(\hat{R})$.
\end{remark}
Now we give the proof of our Main Theorem as a corollary of Theorem \ref{theorem: renewal for hat q}.
\begin{proof}[Proof of the Main Theorem]
Given $a,b\geq1$, $a<b$, $N_1,N_2\in \N$ and $d_j\in\Omega$, $-N_1<j\leq N_2$, we want to write 
\begin{equation}\label{eq: measure a<q_n_L<b etc}
\mu\left(\left\{\alpha:\:a<\frac{q_{n_L}}{L}<b,\,\omega_{n_L+j}=d_j,\,-N_1<j\leq N_2\right\}\right)
\end{equation}
in terms of analogous quantities for $\frac{\hat{q}_{\hat{n}_L}}{L}$.

Denoting by $\nu=\nu_{\hat{n}_L-1}$, we get $\nu_{\hat{n}_L}=\nu+\tau+1$, where $\tau=\tau\left(R^{\hat{n}_L-1}(\hat\omega^+)\right)$. By construction we have $$\hat q_{\hat{n}_L-1}=q_\nu<q_{n_L}\leq q_{\nu+\tau+1}=\hat q_{\hat{n}_L}$$ and therefore $q_{n_L}=q_{\nu+j}$ for some $1\leq j\leq\tau+1$. Notice that, by definition of $\hat n_L$, we have $0<\frac{\hat q_{\hat n_L-1}}{L}<1$.

We distinguish three cases: \textbf{(i)} $j=\tau+1$, \textbf{(ii)} $j=\tau$ and \textbf{(iii)} $1\leq j\leq \tau-1$. Let us remark that in the first case we have $q_{n_L}=\hat q_{\hat n_L}$. For cases (ii) and (iii), by (\ref{eq: recursion p n and q n}), we observe that
\begin{equation}\nonumber
\begin{bmatrix}
2\,k_{\nu+\tau+1}&\xi_{\nu+\tau}&&&&\\
-1&2\,k_{\nu+\tau}&-1&&&\\
&-1&2&\ddots&&\\
&&-1&\ddots&-1&\\
&&&\ddots&2&-1\\
&&&&-1&2
\end{bmatrix}
\cdot
\begin{bmatrix}
q_{\nu+\tau}\\
q_{\nu+\tau-1}\vspace{-.06cm}\\
\vdots\vspace{-.08cm}\\
q_{\nu+j}\vspace{-.08cm}\\
\vdots\vspace{-.1cm}\\
q_{\nu+2}\\
q_{\nu+1}
\end{bmatrix}
=\begin{bmatrix}
q_{\nu+\tau+1}\vspace{.06cm}\\
0\vspace{-.11cm}\\
\vdots\vspace{0cm}\\
0\vspace{-.11cm}\\
\vdots\\
0\\
q_{\nu}
\end{bmatrix}
\end{equation} 
and therefore $q_{\nu+j}=C^{(1)}\,q_{\nu+\tau+1}+C^{(2)}\,q_{\nu}$, with $C^{(i)}=C^{(i)}(j,k_{\nu+\tau},\xi_{\nu+\tau},k_{\nu+\tau+1})$. It can be shown that $0<C^{(1)},C^{(2)}<1$, except for $j=\tau$ and $\xi_{\nu+\tau}=+1$ when we have $-1<C^{(2)}<0$ (see case (ii') below).

We can assume $N_1\geq2$. Indeed the case $N_1=1$ can be recovered by $N_1=2$ considering the sum over all possible values of $\omega_{n_L-1}$. The values of $d_{-1},d_0\in\Omega$ determine the case (i), (ii) or (iii) we are dealing with:
\begin{center}
\begin{tabular}{|c|c|c|}
\hline
$d_{-1}$&$d_0$&case\\
\hline
\hline
$\neq 1^-$&$\in\Omega$&(i)\\
\hline
\multirow{2}{*}{$=1^-$}&$=m^+$&(ii')\\
&$=m^-\:(m\neq1)$&(ii'')\\
\hline
$=1^-$&$=1^-$&(iii)\\
\hline
\end{tabular}
\end{center}
For each case we can rewrite (\ref{eq: measure a<q_n_L<b etc}) as follows:

\textbf{(i)} There exist $N_1',N_2'\in\N$ such that
\begin{eqnarray}
&&\mu\left(\left\{\alpha:\:a<\frac{q_{n_L}}{L}<b,\,\omega_{n_L+j}=d_j,\,-N_1<j\leq N_2\right\}\right)=\nonumber\\
&&=\sum_{c}\mu\left(\left\{\alpha:\:a<\frac{\hat q_{\hat n_L}}{L}<b,\,\sigma_{\hat n_L+j}=c_j,\,-N_1'<j\leq N_2'\right\}\right),\label{eq: case (i)}
\end{eqnarray}
where the sum is taken over those $c=\{c_j\}\in\Sigma^{N_1'+N_2'}$ that, after being coded into the alphabet $\Omega$, are compatible with the $\{d_j\}\in\Omega^{N_1+N_2}$. Notice that, if $c_j=h_j\cdot m_j^{\pm}\in\Sigma$, $-N_1'<j\leq N_2'$, the coding of $c=\{c_j\}$ into the alphabet $\Omega$ gives us a sequence of length $\sum_{j=-N_1'+1}^{N_2'}(h_j+1)$.

\textbf{(ii')} There exist $N_1',N_2'\in\N$ such that
\begin{eqnarray}
&&\mu\left(\left\{\alpha:\:a<\frac{q_{n_L}}{L}<b,\,\omega_{n_L+j}=d_j,\,-N_1<j\leq N_2\right\}\right)=\nonumber\\
&&=\sum_{c}\mu\left(\left\{\alpha:\:\frac{a+C^{(2)}}{C^{(1)}}<\frac{\hat q_{\hat n_L}}{L}<\frac{b}{C^{(1)}},\,\sigma_{\hat n_L+j}=c_j,\,-N_1'<j\leq N_2'\right\}\right),\label{eq: case (ii')}
\end{eqnarray}
where the sum is taken over those $c=\{c_j\}\in\Sigma^{N_1'+N_2'}$ as in (\ref{eq: case (i)}). Notice that $C^{(1)}$ and $C^{(2)}$ depend on $c$ via $k_{\nu+\tau+1}$, $k_{\nu+\tau}$ and $\xi_{\nu+\tau}$.

\textbf{(ii'') \& (iii)}  There exist $N_1',N_2'\in\N$ such that
\begin{eqnarray}
&&\mu\left(\left\{\alpha:\:a<\frac{q_{n_L}}{L}<b,\,\omega_{n_L+j}=d_j,\,-N_1<j\leq N_2\right\}\right)=\nonumber\\
&&=\sum_{c}\mu\left(\left\{\alpha:\:\frac{a}{C^{(1)}}<\frac{\hat q_{\hat n_L}}{L}<\frac{b-C^{(2)}}{C^{(1)}},\,\sigma_{\hat n_L+j}=c_j,\,-N_1'<j\leq N_2'\right\}\right),\label{eq: case (ii'') and (iii)}
\end{eqnarray}
where the sum is taken over those $c=\{c_j\}\in\Sigma^{N_1'+N_2'}$ as in (\ref{eq: case (i)}), with the further constraint that $C^{(2)}<b-a$.

Now, we want to consider the limit as $L\rightarrow\infty$ in (\ref{eq: case (i)}), (\ref{eq: case (ii')}) and (\ref{eq: case (ii'') and (iii)}). 
Let us show that in each case we have uniform convergence of the series.
Denoting
\begin{eqnarray}
\varphi_c^{(1)}(a)=\begin{cases}
     $a$ & \text{case (i)}, \\
      \frac{a+C^{(2)}}{C^{(1)}}& \text{case (ii')},\\
      \frac{a}{C^{(1)}}& \text{case (ii'') or (iii)},\\
\end{cases}\hspace{.5cm}\mbox{and}\hspace{.5cm}\varphi_c^{(2)}(b)=\begin{cases}
     $b$ & \text{case (i)}, \\
      \frac{b}{C^{(1)}}& \text{case (ii')},\\
      \frac{b-C^{(2)}}{C^{(1)}}& \text{case (ii'') or (iii)},\\
\end{cases}\nonumber
\end{eqnarray}
we get
\begin{eqnarray}
\mu\left(\left\{\alpha:\: \varphi^{(1)}_c(a)<\frac{\hat q_{\hat n_L}}{L}<\varphi^{(2)}_c(b),\,\sigma_{\hat n_L+j}=c_j,\,-N_1'<j\leq N_2'\right\}\right)\leq\nonumber\\
\leq\mu\left(\left\{\alpha:\: R^{\hat n_L(\alpha)-1}(\alpha)\in\mathcal{C}'_{N_1',N_2'}\right\}\right)=\mu\left(\mathcal{C}'_{N_1',N_2'}\right),\nonumber
\end{eqnarray}
where $\mathcal{C}'_{N_1',N_2'}:= R^{N_1'-1}\left(\mathcal{C}[c_{-N_1'+1},\ldots,c_0,\ldots,c_{N_2'}]\right)$. Now, if $c_j=h_j\cdot m_j^{\pm}$, by Remark \ref{remark: measure of cylinders}, we obtain the estimate
\begin{eqnarray}
\mu\left(\mathcal{C}'_{N_1',N_2'}\right)=\mu\left(\mathcal{C}[c_{-N_1'+1},\ldots,c_0,\ldots,c_{N_2'}]\right)\leq C_{20}\,\prod_{j=-N_1'+1}^{N_2'}\frac{1}{(4h_j^2+8h_j+3)\,m_j^2},\nonumber
\end{eqnarray}
for some constant $C_{20}>0$.  Now the series of suprema is controlled as follows:
\begin{eqnarray}
&&\hspace{-.6cm}\sum_{\begin{array}{c}c_i=h_j\cdot m_i^\pm\in\Sigma\\j=-N_1'+1,\ldots,N_2'\end{array}}\sup_{L}\left[\mu\left(\left\{\alpha:\: \varphi^{(1)}_c(a)<\frac{\hat q_{\hat n_L}}{L}<\varphi^{(2)}_c(b),\,\sigma_{\hat n_L+j}=c_j,\,-N_1'<j\leq N_2'\right\}\right)\right]\leq\nonumber\\
&&\leq 2\,C_{20}\:
\sum_{h_{-N_1'+1},\ldots,h_{N_2'}\geq0}\:\:\:\sum_{m_{-N_1'+1},\ldots,m_{N_2'}\geq1}
\:\:\:\prod_{j=-N_1'+1}^{N_2'}\frac{1}{(4h^2_j+8h_j+3)\,m_j^2}\leq C_{21},\nonumber
\end{eqnarray}
for some constant $C_{21}>0$.
Now, because of uniform convergence, we can interchange the limit as $L\rightarrow\infty$ and the series in (\ref{eq: case (i)}), (\ref{eq: case (ii')}) and (\ref{eq: case (ii'') and (iii)}) and by Theorem \ref{theorem: renewal for hat q} we get
\begin{eqnarray}
&&\lim_{L\rightarrow\infty}
\mu\left(\left\{\alpha:\:a<\frac{q_{n_L}}{L}<b,\,\omega_{n_L+j}=d_j,\,-N_1<j\leq N_2\right\}\right)=\nonumber\\
&&=\sum_{c} \mathrm{P}_{N_1',N_2'}'\left(\left(\varphi_c^{(1)}(a),\varphi_c^{(2)}(b)\right)\times\{c_{-N_1'+1}\}\times\cdots\times\{c_0\}\times\ \cdots\times\{c_{N_2'}\}\right)<\infty,\nonumber
\end{eqnarray}
concluding thus the proof of our Main Theorem.
\end{proof}

\section*{Appendix A}
Before giving the proof of Proposition \ref{prop: the flow is mixing}, we present, an useful Lemma and its Corollary. 
\begin{lem}\label{lem: backward convergents}
If $\left\{p_i/q_i\right\}$ are the ECF-convergents of $\alpha=\les(k_1,\xi_1),(k_2,\xi_2),\ldots\res$, then for any $m\geq2$
\begin{eqnarray}
\les(k_m,\xi_{m-1}),(k_{m-1},\xi_{m-2}),\ldots,(k_2,\xi_1),(k_1,*)\res&=&\frac{q_{m-1}}{q_m},\nonumber\\
\les(k_m,\xi_{m-1}),(k_{m-1},\xi_{m-2}),\ldots,(k_3,\xi_2),(k_2,*)\res&=&\frac{p_{m-1}}{p_m}.\nonumber
\end{eqnarray}
\end{lem}
\begin{cor}\label{cor: backward ecf in terms of forward convergents}
Given $\left\{(k_n,\xi_n)\right\}_{n\in\Z}$, let us consider $m\geq2$, $\beta\in\X$ and $\gamma\in\left[-1,1\right]\smallsetminus\Q$ such that 
\begin{eqnarray}
\beta&\hspace{-.2cm}=&\hspace{-.2cm}\les(k_m,\xi_{m-1}),(k_{m-1},\xi_{m-2}),\ldots,(k_2,\xi_1),(k_1,\xi_0),(k_0,\xi_{-1}),\ldots\res,\nonumber\\
\gamma&\hspace{-.2cm}=&\hspace{-.2cm}\les(0,\xi_0);(k_0,\xi_{-1}),(k_{-1},\xi_{-2}),(k_{-2},\xi_{-3}),\ldots\res,\nonumber
\end{eqnarray}
and let $\left\{p_i/q_i\right\}$ be the ECF-convergents of $\les(k_1,\xi_1),(k_2,\xi_2),(k_3,\xi_3)\ldots\res$.
Then 
\begin{equation}\label{eq: backward ecf in terms of forward convergents}
\beta=\frac{q_{m-1}+p_{m-1}\cdot\gamma}{q_{m}+p_{m}\cdot\gamma}.
\end{equation}
\end{cor}
We skip the proofs of Lemma \ref{lem: backward convergents} and Corollary \ref{cor: backward ecf in terms of forward convergents} since they can be recovered, \emph{mutatis mutandis}, from proofs of the analogous results for Euclidean expansions. Now we are ready to give the following
\begin{proof}[Proof of Proposition \ref{prop: the flow is mixing}]\bianco\\
This proof imitates the one given in \cite{Dinaburg-Sinai} (and reviewed in \cite{Sinai-Ulcigrai07}) concerning Euclidean expansions.
In order to prove the statement, we shall perform two steps: 1) construct the global stable and unstable foliations, 2) prove that they form a non-integrable pair (see e.g. \cite{Anosov1967}, Chapter I or \cite{Dinaburg-Sinai}, \S 2.4).
Their non integrability will imply that the Pinsker partition \cite{Sinai-Topics} is trivial and hence $\{\Phi_t\}_t$ is a K-flow and, in particular, it is mixing.

Given $(\hat \omega^{(0)},y_0)\in D_\Phi$, let us construct the local stable and unstable leaves through it, denoted by $\Gamma_{\mbox{\tiny{loc}}}^{(s)}(\hat \omega^{(0)},y_0)$ and  $\Gamma_{\mbox{\tiny{loc}}}^{(u)}(\hat \omega^{(0)},y_0)$ respectively. As a reference, see e.g. \cite{Sinai-Topics}.\\
Since the roof function $\psi(\hat\omega)$ depends only on $(\ldots,\sigma_{-1},\sigma_{0};\sigma_{1},\sigma_2)$, $\sigma_i\in\Sigma$, 
it is clear that the local unstable leaf is given by a piece of segment in the $\hat\omega^+$-direction:
\begin{eqnarray}
\Gamma_{\mbox{\tiny{loc}}}^{(u)}(\hat \omega^{(0)},y_0)\subset\left\{(\hat\omega,y_0):\:\:(\hat R^2\hat\omega)^-=(\hat R^2\hat\omega^{(0)})^-\right\}\subset\left\{(\hat\omega,y_0):\:\:\hat\omega^-=(\hat\omega^{(0)})^-\right\}.\label{eq: local unstable leaf}
\end{eqnarray}
The local stable leaf $\Gamma_{\mbox{\tiny{loc}}}^{(s)}(\hat \omega^{(0)},y_0)$ is given by those pairs $(\hat\omega,y)$ satisfying 
\begin{equation}\label{eq: local stable leaf}
\begin{cases}
     \hat\omega^+=(\hat\omega^{(0)})^+, \\
     y=y_0+\displaystyle{\log\left(\frac{2k_1+(\hat\omega^{(0)})^-}{2k_1+\hat\omega^-}\right)
+\log\left(\frac{1+(\hat\omega^{(0)})^+\cdot\hat\omega^-}{1+(\hat\omega^{(0)})^+\cdot(\hat\omega^{(0)})^-}\right)}.
\end{cases}
\end{equation}
In order to see this, let us denote $(\hat\omega^{(t)},y_t)=\Phi_t(\hat\omega^{(0)},y_0)$ and consider a small segment in the $\hat\omega^-$-direction through it:
\begin{equation}\nonumber
\Gamma^t_{\delta_t}=\left\{(\hat\omega,y_t):\:\:\hat\omega^+=(\hat\omega^{(t)})^+,\:\left|\hat\omega^--(\hat\omega^{(t)})^-\right|<\delta_t\right\},
\end{equation}
where $\delta_t$ is chosen sufficiently small so that
\begin{eqnarray}
\Phi_{-t}\hspace{-.08cm}\left(\Gamma^t_{\delta_t}\right)\hspace{-.08cm}\subset\hspace{-.08cm}\Big\{\hspace{-.05cm}(\hat\omega,y):\:\hat\omega^+=(\hat\omega^{(0)})^+,
\:\left|y-y_0\right|<\delta,\:
0<y_0-\delta<y<\varphi(\hat\eta)-\delta \Big\}\nonumber
\end{eqnarray}
for some $\delta>0$.\\ Now, if $(\hat\omega,y)\in\Phi_{-t}\left(\Gamma^t_{\delta_t}\right)$, then by construction $r(\hat\omega,t)=r(\hat\omega^{(0)},t)=:r(t)$ and, from the definition (\ref{eq: action special flow Phi_t}) of the special flow, $y-S_{r(t)-1}(\psi,\hat R)(\hat\omega)=y_t-t=y_0-S_{r(t)-1}(\psi,\hat R)(\hat\omega^{(0)})$. From (\ref{def: roof function psi}) and (\ref{eq: backward ecf in terms of forward convergents}), because of telescopic cancellations, we get
\begin{eqnarray}
S_{r(t)-1}(\psi,\hat R)(\hat\omega)=\sum_{l=2}^{\nu}\log\les(k_l,\xi_{l-1}),(k_{l-1},\xi_{l-2}),\ldots\res^{-1}=\frac{q_\nu+p_\nu\cdot\hat\omega^-}{q_1+p_1\cdot\hat\omega^-},\nonumber
\end{eqnarray}
where $\nu=\nu_{r(t)-1}$ and $\left\{p_i/q_i\right\}$ are the ECF-convergents of $(\hat\omega^{(0)})^+$.
Similarly we find 
\begin{eqnarray}
S_{r(t)-1}(\psi,\hat R)(\hat\omega^{(0)})=\frac{q_\nu+p_\nu\cdot(\hat\omega^{(0)})^-}{q_1+p_1\cdot(\hat\omega^{(0)})^-}\nonumber
\end{eqnarray}
and therefore, recalling that $p_1=1$ and $q_1=2k_1$, 
\begin{eqnarray}
y&=&y_0+S_{r(t)-1}(\psi,\hat R)(\hat\omega)-S_{r(t)-1}(\psi,\hat R)(\hat\omega^{(0)})=\nonumber\\
&=&y_0+\log\left(\frac{2k_1+(\hat\omega^{(0)})^-}{2k_1+\hat\omega^-}\right)
+\log\left(\frac{1+\frac{p_\nu}{q_\nu}\,\hat\omega^-}{1+\frac{p_\nu}{q_\nu}\,(\hat\omega^{(0)})^-}\right).\label{eq: proof of equation of local stable leaf}
\end{eqnarray}
As $t\rightarrow\infty$, also $\nu\rightarrow\infty$ and $\frac{p_\nu}{q_\nu}\rightarrow(\hat\omega^{(0)})^+$; therefore (\ref{eq: proof of equation of local stable leaf}) leads us to (\ref{eq: local stable leaf}).\\ 
The global unstable and stable leaves are obtained as
\begin{eqnarray}
\hspace{.05cm}\Gamma^{(u)}(\hat\omega^{(0)},y_0)=\bigcup_t\Phi_t\Gamma^{(u)}_{\mbox{\tiny{loc}}}(\hat\omega^{(-t)},y_{-t}),\hspace{.45cm}\Gamma^{(s)}(\hat\omega^{(0)},y_0)=\bigcup_t\Phi_{-t}\Gamma^{(u)}_{\mbox{\tiny{loc}}}(\hat\omega^{(t)},y_{t}).\nonumber
\end{eqnarray}
Let us consider a sufficiently small neighborhood $\mathcal{U}^{(0)}\subset D_\Phi$ of $(\hat\omega^{(0)},y_{0})$. In order to prove the non integrability of the stable and unstable foliations, it is enough to show that, for a set of positive measure of $(\hat\omega,y)\in\mathcal{U}^{(0)}$, $(\hat\omega,y)$ can be connected to $(\hat\omega^{(0)},y_0)$ through a polygonal made of segments of stable and unstable leaves. This is achieved in particular if there exist $(\hat\omega',y'), (\hat\omega'',y'')\in\mathcal{U}^{(0)}$, 
such that $(\hat\omega',y')\in\Gamma^{(s)}(\hat\omega^{(0)},y_0)$, $(\hat\omega'',y'')\in\Gamma^{(u)}(\hat\omega',y')$ and $(\hat\omega,y)\in\Gamma^{(s)}(\hat\omega'',y'')$. Using the explicit local parametrizations (\ref{eq: local unstable leaf}-\ref{eq: local stable leaf}), one can check that these points exist as soon as one can find $(\hat\omega')^-$ and $y'$ such that
\begin{equation}\label{eq: existence of points omega' omega''}
\begin{cases}
      y'=y_0+\displaystyle{\log\left(\frac{2k_1+(\hat\omega^{(0)})^-}{2k_1+(\hat\omega')^-}\right)+\log\left(\frac{1+(\hat\omega^{(0)})^+\cdot(\hat\omega')^-}{1+(\hat\omega^{(0)})^+\cdot(\hat\omega^{(0)})^-}\right)},\\
      y=y'+\displaystyle{\log\left(\frac{2k_1+(\hat\omega')^-}{2k_1+\hat\omega^-}\right)+\log\left(\frac{1+\hat\omega^+\cdot\hat\omega^-}{1+\hat\omega^+\cdot(\hat\omega')^-}\right)}.
\end{cases}
\end{equation}
A direct computation shows that equations (\ref{eq: existence of points omega' omega''}) in the unknowns $(\hat\omega')^-$ and $y'$ can be solved when
\begin{eqnarray}
\frac{\hat\omega^+\left(2k_1+\hat\omega^-\right)\,e^y}{1+\hat\omega^-\cdot\hat\omega^+}\neq\frac{(\hat\omega^{(0)})^+\left(2k_1+(\hat\omega^{(0)})^-\right)\,e^{y_0}}{1+(\hat\omega^{(0)})^-\cdot(\hat\omega^{(0)})^+}\label{eq: equations can be solved when 1},
\end{eqnarray}
and (\ref{eq: equations can be solved when 1}) holds true for $(\hat\omega,y)$ in a subset of $\mathcal{U}^{(0)}$ with positive measure, as desired.
\end{proof}
 
\section*{Acknowledgments}
I would like to thank my advisor Prof. Ya. G. Sinai for his constant guidance and invaluable encouragement. Thanks also to C. Ulcigrai for many fruitful conversations and to the anonymous referee for his/her thorough remarks. Lastly, I would like to dedicate this work to the memory of Riccardo Venier.
%\nocite{*}
\addcontentsline{toc}{chapter}{Bibliography}
\bibliographystyle{plain}
\bibliography{renewal-type-revised-bibliography}
\end{document}